\documentclass[msom,sglanonrev]{informs4}

\usepackage{graphicx}
\usepackage{subfig}
\graphicspath{{figs/}}

\DoubleSpacedXI
\EquationsNumberedThrough
\TheoremsNumberedThrough
\ECRepeatTheorems

\begin{document}

\RUNAUTHOR{Lin and Zhang}
\RUNTITLE{AI Review Before Human Approval}

\TITLE{When Should AI Review Precede Human Approval of Agent-Generated Code?}

\ARTICLEAUTHORS{%
\AUTHOR{Ruihan Lin, Jiheng Zhang}
\AFF{Department of Industrial Engineering and Decision Analytics,
The Hong Kong University of Science and Technology, Hong Kong \\
\EMAIL{rlinah@connect.ust.hk, jiheng@ust.hk}}
}

\ABSTRACT{
\textbf{Problem definition:} Software teams are adopting language-model
agents that draft complete pull requests from real issues. A generated
change, however, still cannot merge until a human maintainer approves it.
Teams can place an optional AI reviewer in front of that gate to filter
defective work before a person sees it, but the reviewer uses inference
capacity and its false rejections send merge-ready changes back for another
agent run. We ask how much of this review to use, and for which task
categories.
\textbf{Methodology/results:} A queueing analysis shows that the review
decision is a staffing map, not a fixed ranking. For one category the
optimal review fraction passes through four phases as maintainer capacity
grows, and in one phase an informative reviewer is deliberately left idle
because its rework displaces new work. For many categories the map changes
only at finitely many staffing breakpoints, leaves at most three categories
partially served or split across routes, and cannot be produced by ranking
categories once. The map is implementable: Quota Tracking attains the
planning bound in the fluid limit, and in a calibrated study its gap to
that bound shrinks with system size while fixed and single-index rules keep
theirs.
\textbf{Managerial implications:} AI review is a capacity decision as much
as an accuracy decision. False acceptance and false rejection load
different resources, so a single accuracy score cannot say how much to
review or whom to review; both answers move when the bottleneck moves.
Operators should estimate both error types and each category's service
requirements, compute the staffing breakpoints, and change review intensity
and category priority at those points.
}


\maketitle

\def\MainIncludesAppendix{}
\section{Introduction}

Software teams increasingly rely on language-model systems to draft the
change that used to take most of the calendar time between an issue and a
pull request, so what the team now waits for is a maintainer's approval
rather than the writing of the code.  SWE-bench
\citep{jimenez2024swebench} asks whether such a system can resolve a real
GitHub issue, and agent designs such as SWE-agent
\citep{yang2024sweagent}, AutoCodeRover \citep{zhang2024autocoderover}, and
Agentless \citep{xia2025agentless} have raised those resolution rates
quickly.  Separately, three field experiments find that developers who use
an AI coding assistant complete about 26\% more tasks
\citep{cui2026effects}.  The merge itself has not been automated: a
maintainer still reads the diff and clicks approve, because that step
remains an ownership decision rather than a generation step.

The resulting pile-up is what makes an AI reviewer attractive.
SWE-Review \citep{wang2026swereview} puts an independent agent on a
generated pull request and still records both false approvals and false
rejections.  Even the generator is not a fixed quality: the same agent on
the same task succeeds in some runs and fails in others
\citep{bjarnason2026randomness}.  A change the reviewer lets through still
occupies a maintainer.  A change it sends back occupies another
coding-agent run.  That extra run is cheap if agents are waiting for work
and costly if they are already busy writing the next change.

A natural pair of operating defaults is then to treat review as a score.
One default leaves an informative reviewer on for every change, as if
screening were free.  The other ranks categories once by quality or by a
single error rate and sends review capacity to the top of the list.
Neither asks which resource the next review will tax, or whether that
resource is the one that is scarce today.

The workflow itself suggests a different description.  Each category can
reach the maintainer on a direct path or after an AI screen.  The screened
path buys maintainer time by intercepting defects, and pays for it with
inference and with regeneration when good work is rejected.  Categories
that look accurate after review may be slow to generate, expensive to
screen, or heavy on false rejection; categories that look sloppy may be
cheap to regenerate.  Work that waits too long can become obsolete as the
repository moves on.  The operating objective is approved throughput, not
reviewer accuracy.

These observations lead to three questions.  First, as maintainer capacity
changes, when should the operator review all generated work, review a
fraction, or bypass AI review entirely?  Second, with many categories, which
should receive scarce review capacity, and can a ranking of categories that
does not depend on capacity implement the optimum?  Third, can a steady-state allocation be turned into
event-level routing and service decisions without losing its throughput
benefit in the stochastic system?

We make three contributions, all properties of one staffing map.  First,
the map says how much to review.  In a single-category benchmark the
optimal review fraction moves through four phases as maintainer capacity
grows: full review, reviewer saturation with partial bypass, active
throttling once generation binds, and complete bypass when human approval
is abundant.  The third phase is the main insight.  An informative
reviewer is deliberately left idle because the regeneration it causes
displaces new work.  The thresholds are explicit in service rates, error
rates, and capacities.

Second, the map says whom to review.  Across any number of categories the
allocation is piecewise linear in maintainer staffing and changes only at
finitely many breakpoints.  Between breakpoints at most three categories
are marginal, being partially served or split across routes; all others
are fully served or unused.  Even when categories differ only in their
reviewers, no ranking fixed across staffing levels is optimal: scarce
maintainer time ranks reviewers by the false acceptances they let through,
scarce generation by the merge-ready work they send back, and the reviewed
category can change with staffing alone.

Third, the map is implementable.  Quota Tracking turns every planned
fraction into a deterministic cumulative quota, corrects an arbitrary
bounded initial state, and attains the planning bound under a fluid-first
limit.  A calibrated study then measures the cost of ignoring the map:
the value of adapting review to the bottleneck, the regret from ranking
categories once, the effect of calibration uncertainty, and how fast the
implemented map approaches the fluid optimum as the system grows.

\subsection{Literature Review}
\label{subsec:lit_review}

Three streams of work touch the problem, and the paper sits where they
meet.

\paragraph{Coding agents and code review.}
The agent systems and field evidence above measure whether a generated
change resolves an issue, and they treat quality as a property of attempts
rather than of tasks.  They do not ask how many changes a team can approve
per unit time when generation is fast and human approval is not.  On the
review side, \citet{bacchelli2013expectations} document that finding
defects is the primary purpose of modern code review, and
\citet{tufano2021automating} and \citet{li2022automating} train models to
draft review comments.  SWE-Review \citep{wang2026swereview} closes the
loop with an independent agent that reviews a generated pull request and
requests revisions, and it reports the agent's false-rejection and
false-acceptance rates by generator tier.  We take those measured error
rates as inputs and study the throughput question.

\paragraph{Imperfect AI evaluation and human--AI collaboration.}
Language models are increasingly used as judges of other models' output.
\citet{zheng2023judging} find that strong judges agree with human raters
about as often as humans agree with one another, but also document position
and verbosity biases; \citet{wang2024notfair} and \citet{koo2024cognitive}
show that such biases can flip a verdict; \citet{tan2024judgebench} show
that judge accuracy collapses on hard reasoning tasks; and
\citet{bavaresco2025llms} find that reliability varies widely across twenty
evaluation tasks.  Specialized evaluators \citep{zhu2023judgelm,
kim2024prometheus} and structured prompting \citep{liu2023geval} narrow but
do not close the gap, and \citet{gu2024survey} survey the field.  An AI
reviewer's two error rates are thus task dependent and must both enter any
operating model.  A separate literature
decides who should make a decision.  Learning-to-defer models train a
classifier together with a rule for passing hard cases to a human
\citep{madras2018predict,mozannar2020consistent}, medical applications
predict when a second opinion is needed \citep{raghu2019uncertainty}, and
\citet{donahue2022collaboration} characterize when the combination beats
either party alone.  \citet{fugener2022cognitive} show experimentally that
humans use AI advice poorly when they cannot tell good advice from bad, and
\citet{zhong2025optimal} derive the optimal integration of human, machine,
and generative-AI judgment.  Closest to us, \citet{lykouris2024learning}
and \citet{lee2024design} place deferral inside a queue, so that sending a
case to the human costs waiting time.  In all of these models the human is
optional and the AI can be terminal.  In ours the human is mandatory and
the AI is an optional filter before the human, so the question is no longer
whether the AI should decide, but whether an intermediate AI opinion is
worth its own capacity and the regeneration its mistakes cause.

\paragraph{Rework networks and fluid control.}
Queues with feedback are classical.  \citet{adve1994relationship} show that
deterministic feedback reduces variability relative to Bernoulli feedback in
an M/G/1 queue, an observation our routing rule exploits;
\citet{so1995optimal} derive the optimal operating policy for a bottleneck
with random rework, \citet{lee1997design} size a feedback buffer, and
re-entrant lines with inspection stations model semiconductor fabrication
\citep{kumar1993reentrant,narahari1996modeling,chen1988empirical}.  In
services, Erlang-R captures patients who return to the physician
\citep{yom2014erlang}, and \citet{chan2014use} show that speeding up service
can backfire when it raises the return probability.  Fluid models supply
the analytical tools: \citet{dai1995positive} establishes stability of
multiclass networks through their fluid limits, \citet{whitt2006fluid} and
\citet{zhang2013fluid} develop fluid models for many-server queues with
abandonment, and \citet{chen1993dynamic} schedule a multiclass fluid
network.  Index policies derived from fluid or diffusion approximations,
including the \(c\mu/\theta\) rule \citep{atar2010cmu}, its generalizations
\citep{long2020dynamic,long2024generalized}, thin-stream control
\citep{ata2006dynamic}, and dynamic matching \citep{gurvich2015dynamic,
hu2022dynamic}, translate a static allocation into event-level decisions.
A related fluid analysis of large-scale LLM inference
\citep{lin2026inference} gates and routes heterogeneous prefill and decode
jobs that contend for shared GPUs; the scarce resource is compute, and
routing does not send finished work back for regeneration.
In all of them the probability of rework is a parameter of the process.  In
ours it is a decision, because routing a change to AI review raises the
chance that a merge-ready change is sent back, in exchange for filtering
defective ones.

\paragraph{Positioning.}
Each stream has one of the three pieces the problem needs and lacks the
other two.  Evaluation research measures reviewer errors but not their
effect on capacity.  Collaboration models allocate decisions but treat the
human as optional and ignore the workload that AI errors create upstream.
Rework and fluid models have the machinery but treat rework as exogenous.
As coding agents make generation cheap, human approval becomes the binding
constraint, and two operating defaults then fail.  Always-on AI review
treats an informative filter as free, and a ranking of categories by
quality or by a single error statistic does not move when the bottleneck
moves.  Both defaults are wrong: review intensity and the reviewed
category must jump at staffing breakpoints, and the resulting map is
sparse enough to implement.

The paper proceeds as follows.  Section~\ref{sec:problem} formulates the
network and the control problem.  Section~\ref{sec:fluid_model} derives the
steady-state capacity-allocation problem.  Section~\ref{sec:canonical}
solves it for one category and characterizes the multiclass staffing
response.  Section~\ref{sec:stoch_control} presents Quota Tracking and
its guarantee, Section~\ref{sec:numerical} reports the calibrated study, and
Section~\ref{sec:conclusion} concludes.  The Appendix contains the
stochastic state equations, the fluid dynamics, and proof sketches of the
planning results.  Complete proofs of
Theorems~\ref{thm:fluid-limit} and~\ref{thm:asymptotic_optimality}
and the analysis of Quota Tracking appear in the Online Supplement.

\section{Problem Formulation}
\label{sec:problem}

We formulate the coding workflow as a multiclass service network with rework
and task obsolescence.  This section describes the workflow, maps reviewer
errors into flows, and specifies the stochastic primitives, the admissible
controls, and the objective of the control problem.

\begin{figure}[htbp]
\centering
\includegraphics[width=0.84\textwidth]{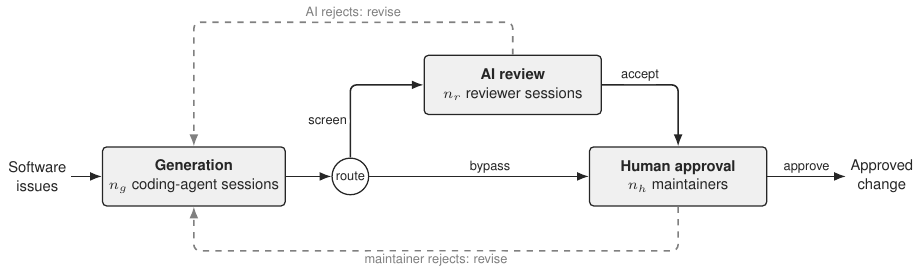}
\caption{Workflow architecture.  Solid arrows carry forward flow; dashed
arrows return rejected changes to generation for revision.  A generated
change is either screened by AI review or bypasses it, and every change
still requires human approval before it can merge.}
\label{fig:network}
\end{figure}

\subsection{Workflow}
\label{subsec:pf_workflow}

Figure~\ref{fig:network} shows the three resource pools.  Software issues
arrive exogenously and wait in the generation queue for a coding-agent run.
When a run completes, a router sends the proposed change either to the
AI-review queue or directly to the human-approval queue; changes that pass
AI review also join the human-approval queue.  A task is completed only when
a human maintainer approves the change.  Two rework channels (dashed
arrows) return work to the generation queue: the AI reviewer may request a
revision and a maintainer may reject the change.  A false rejection thus
sends a merge-ready change through another coding-agent run, and a false
acceptance uses review capacity without reducing human work.  The
subscripts \(g,r,h\) denote generation, AI review, and human approval.

\subsection{Quality and Error Model}
\label{subsec:pf_quality}

Consider \(I\) task classes indexed by \(i\in\mathcal I:=\{1,\dots,I\}\).
Let \(r_i>0\) be the value of an approved class-\(i\) change.  A
coding-agent output for class \(i\) is merge-ready with probability
\(1-\alpha_i\) and defective otherwise.  A maintainer detects defects, so on
the direct path a change is rejected with probability \(\alpha_i\).
Attempts are independent within a class, so after a requested revision the
next output has the same error probability \(\alpha_i\).

Conditional on latent quality, the AI reviewer makes two kinds of error.
Its Type~I error \(\beta_i^{(I)}\) is the probability that it requests
revisions to a merge-ready change (false rejection); its Type~II error
\(\beta_i^{(II)}\) is the probability that it accepts a defective change
(false acceptance).
\begin{align}
p_{\mathrm{pass},i} &:= (1-\alpha_i)(1-\beta_i^{(I)}) + \alpha_i\beta_i^{(II)}, \label{eq:pf_ppass}\\
p_{\mathrm{rej},i}  &:= (1-\alpha_i)\beta_i^{(I)} + \alpha_i(1-\beta_i^{(II)}), \label{eq:pf_prej}
\end{align}
and an accepted change is merge-ready with posterior probability
\begin{equation}
q_{\mathrm{acc},i}
:=\mathbb{P}(\text{merge-ready}\mid \text{AI reviewer accepts})
=\frac{(1-\alpha_i)(1-\beta_i^{(I)})}{p_{\mathrm{pass},i}}.
\label{eq:pf_qacc}
\end{equation}
On the screened path a maintainer therefore rejects with probability
\(1-q_{\mathrm{acc},i}\).  These three quantities carry all reviewer
information used below.  The share of reviewed work that proceeds to human
approval is \(p_{\mathrm{pass},i}\), the share returned for regeneration
is \(p_{\mathrm{rej},i}\), and the quality of what proceeds is
\(q_{\mathrm{acc},i}\).

AI review changes the composition of work reaching maintainers.  Among
directly routed changes the merge-ready share is \(1-\alpha_i\); among
AI-accepted changes it is \(q_{\mathrm{acc},i}\).  We call the reviewer
\emph{informative} for class \(i\) when \(q_{\mathrm{acc},i}>1-\alpha_i\),
that is, when screening raises quality at the human gate.  If the inequality
fails, screening class \(i\) cannot increase approved throughput, and an
optimal plan can bypass review for that class.

Independent attempts are a simplification: a reviewer's report may guide
the revision, so that an output revised after an AI rejection is
merge-ready with a higher probability \(1-\alpha_i'\).  Such
stage-dependent quality leaves Sections~\ref{sec:fluid_model}
and~\ref{sec:canonical} intact, because the planning program depends on
each class only through the expected generation, review, and maintainer
time per approved change, which remain finite constants with closed forms
(Online Supplement, Section~EC.1).  It can change one comparison: the
throttling phase of Proposition~\ref{prop:four_phase} exists exactly when
the screened route still consumes more generation per approved change than
the direct route.  Under independent attempts this holds whenever
\(\beta_i^{(I)}>0\); with the calibration of
Section~\ref{sec:numerical} it is reversed for the low-quality tier once
revised outputs resolve more than 42\% of issues, close to the gain the
source reports for one review-guided revision, and survives for the
medium-quality tier unless they resolve more than 86\%.  We keep
independent attempts as the baseline.

\paragraph{Imperfect maintainers.}
The maintainer above is a perfect quality gate, so AI review adds value
only by saving capacity.  Suppose instead that a maintainer approves a
defective change with probability \(\gamma_i\) and each escaped defect
costs \(c_i\).  Every planning result below survives after two
substitutions (Online Supplement, Section~EC.1): the reward per unit of
generation becomes
\(r_i(1-\alpha_i)-c_i\alpha_i\gamma_i\), and the false-rejection rate
\(\beta_i^{(I)}\) in the objective is replaced by the net generation cost
of review once the escaped defects it prevents are credited,
\(\tilde\beta_i:=[r_i(1-\alpha_i)\beta_i^{(I)}-c_i\alpha_i\gamma_i
(1-\beta_i^{(II)})]/[r_i(1-\alpha_i)-c_i\alpha_i\gamma_i]\).  The
throttling and bypass phases of Proposition~\ref{prop:four_phase} exist
exactly when \(\tilde\beta_i>0\), that is, when the merge-ready throughput
lost to false rejection outweighs the defect cost review prevents;
otherwise review stays at reviewer saturation however many maintainers are
added.  We set \(\gamma_i=0\) below so that every effect of review is a
capacity effect; \(\tilde\beta_i\) is the quantity to estimate when it is
not.

\subsection{Stochastic Model and Control Problem}
\label{subsec:pf_stochastic}

We consider a sequence of systems indexed by \(n\).  Class-\(i\) tasks
arrive as a Poisson process of rate \(n\lambda_i\).  The generation,
AI-review, and human-approval pools contain \(\lfloor nn_g\rfloor\),
\(\lfloor nn_r\rfloor\), and \(\lfloor nn_h\rfloor\) servers.  Service times
are independent and exponential with class-dependent rates
\(\mu_{g,i},\mu_{r,i},\mu_{h,i}\).  A task waiting at the generation,
AI-review, direct-approval, or post-review-approval node becomes obsolete at
rate \(\theta_{g,i},\theta_{r,i},\theta_{h_d,i}\), or \(\theta_{h_s,i}\);
we write \(\theta_i:=\theta_{g,i}\).  Obsolescence clocks are independent
of latent quality and stop during service.  Each completed attempt draws
its latent quality, after which AI and human outcomes follow the
probabilities in Section~\ref{subsec:pf_quality}.  Arrival, service,
outcome, and obsolescence primitives are mutually independent, and all
error parameters, service rates, and normalized capacities are fixed as
\(n\) grows.

\begin{assumption}[Primitive regularity]\label{ass:regular_params}
For every class, \(\lambda_i,\mu_{g,i},\mu_{r,i},\mu_{h,i}\), and
\(\theta_{i}\) are positive; the downstream obsolescence rates
\(\theta_{r,i},\theta_{h_d,i},\theta_{h_s,i}\) are finite and strictly
positive; \(\alpha_i\in(0,1)\), \(\beta_i^{(I)}<1\), and
\(p_{\mathrm{pass},i},p_{\mathrm{rej},i}>0\).  All normalized capacities
are finite and positive.
\end{assumption}
Positive downstream obsolescence means that an initial backlog at review or
approval cannot remain in the system forever; it is the condition under
which the control policy of Section~\ref{sec:stoch_control} recovers from
any bounded initial state before pooling begins.

Let \(Q_{\ell,i}^n(t)\) and \(X_{\ell,i}^n(t)\) denote the class-\(i\)
waiting and in-service counts at node \(\ell\in\{g,r,h_d,h_s\}\).  For the
in-service counts we also write \(X_i^n,Y_i^n,Z_{i,d}^n,Z_{i,s}^n\) at the
four nodes, and we let \(C_i^n(t)\) count class-\(i\) approvals by time
\(t\).  Appendix~\ref{app:state_equations} states the balance equations
and the Poisson random-time-change representation.

\begin{assumption}[Initial states]\label{ass:fluid-init}
The fluid-scaled initial state \(\bigl(Q_{\ell,i}^n(0),X_{\ell,i}^n(0)\bigr)/n\)
converges in expected \(\ell_1\) norm to a deterministic finite vector.
\end{assumption}

A policy chooses which waiting task enters service when a server becomes
available and the route of each generation completion.  Policies are
nonanticipative, nonpreemptive, event-driven, capacity-feasible, and
first-come-first-served within each class; the number of admissions and
routing decisions made at one event is bounded independently of \(n\), which
excludes fluid-scale batch moves.  Let \(\Pi^n\) denote this policy class.
The operator maximizes weighted approved throughput:
\begin{equation}
\max_{\boldsymbol{\pi}^n\in\Pi^n}
\liminf_{T\to\infty}\frac{1}{nT}
\mathbb E^{\boldsymbol{\pi}^n}\!\left[\sum_{i\in\mathcal I}r_iC_i^n(T)\right],
\label{eq:pf_obj}
\end{equation}
where \(r_i\) is the value of an approved class-\(i\) change.

\section{Steady-State Capacity Allocation}
\label{sec:fluid_model}

Scaling populations and cumulative flows by \(n\) yields a deterministic
fluid model.  The next theorem records that every subsequential limit of
the stochastic system is a trajectory of that model; we then plan
capacity with its stationary points.

\begin{theorem}[Fluid limit]
\label{thm:fluid-limit}
Fix \(T>0\).  Under Assumptions~\ref{ass:regular_params}
and~\ref{ass:fluid-init} and any admissible policy sequence, the
fluid-scaled state and flow process is \(C\)-tight on \([0,T]\).  Every
subsequential weak limit is almost surely continuous and satisfies the
deterministic balances \eqref{eq:fluid-balance}--\eqref{eq:fluid-rates}
and the capacity constraints \eqref{eq:fluid-cap}, with the initial state
of Assumption~\ref{ass:fluid-init}.
\end{theorem}

The fluid equations are stated in Appendix~\ref{app:fluid_model}; the
proof is in Section~EC.2 of the Online Supplement.

\paragraph{Full stationary program.}
The complete planning problem chooses nonnegative waiting levels
\(\hat q_{\ell,i}\), service levels
\(\hat x_i,\hat y_i,\hat z_{i,d},\hat z_{i,s}\), admission rates
\(\hat u_{\ell,i}\), and routing levels \(0\le\hat v_i\le\hat x_i\):
\(\hat x_i\) is generation in service, \(\hat y_i\) is AI review,
\(\hat z_{i,d}\) and \(\hat z_{i,s}\) are maintainer occupancy on the
direct and screened routes, and \(\hat v_i\) is the generation occupancy
routed to review.  It maximizes weighted approved throughput subject to
the stationary node balances and the three resource capacities:
\begin{subequations}\label{eq:full_fluid_program}
\begin{equation}
V_{\mathrm{full}}(n_h):=\max
\sum_i r_i\mu_{h,i}\bigl[(1-\alpha_i)\hat z_{i,d}
+q_{\mathrm{acc},i}\hat z_{i,s}\bigr],
\label{eq:full_fluid_obj}
\end{equation}
\begin{align}
\lambda_i+\mu_{r,i}p_{\mathrm{rej},i}\hat y_i
+\mu_{h,i}\alpha_i\hat z_{i,d}
+\mu_{h,i}(1-q_{\mathrm{acc},i})\hat z_{i,s}
&=\hat u_{g,i}+\theta_i\hat q_{g,i},&
\hat u_{g,i}&=\mu_{g,i}\hat x_i,\nonumber\\
\mu_{g,i}\hat v_i&=\hat u_{r,i}+\theta_{r,i}\hat q_{r,i},&
\hat u_{r,i}&=\mu_{r,i}\hat y_i,\nonumber\\
\mu_{g,i}(\hat x_i-\hat v_i)
&=\hat u_{h_d,i}+\theta_{h_d,i}\hat q_{h_d,i},&
\hat u_{h_d,i}&=\mu_{h,i}\hat z_{i,d},\nonumber\\
\mu_{r,i}p_{\mathrm{pass},i}\hat y_i
&=\hat u_{h_s,i}+\theta_{h_s,i}\hat q_{h_s,i},&
\hat u_{h_s,i}&=\mu_{h,i}\hat z_{i,s},
\label{eq:fluid-ss-balances}
\end{align}
\begin{equation}
\sum_i\hat x_i\le n_g,\qquad
\sum_i\hat y_i\le n_r,\qquad
\sum_i(\hat z_{i,d}+\hat z_{i,s})\le n_h.
\label{eq:fluid-ss-cap}
\end{equation}
\end{subequations}
The four rows are generation, AI review, direct approval, and post-review
approval.  Revision returns enter on the left-hand side; obsolescence
appears as \(\theta\hat q\) on the right.

\paragraph{Two processing routes.}
For maximum throughput, queues and revision loops enter
\eqref{eq:full_fluid_program} only through the expected capacity consumed
per approved change.  Put \(m_i=1-\alpha_i\) and \(t_i=1-\beta_i^{(I)}\),
the probabilities that an attempt is merge-ready and that a merge-ready
attempt survives AI review.
Let \(f_i^D\) and \(f_i^S\) be the approval rates on the direct and
screened routes.  One direct approval earns \(r_i\) and consumes
\begin{equation}
a_i^D=\left(\frac{1}{\mu_{g,i}m_i},\,0,\,
\frac{1}{\mu_{h,i}m_i}\right)
\label{eq:activity_direct_vector}
\end{equation}
units of generation, AI-review, and human-approval capacity, because on
average \(1/m_i\) attempts are generated and inspected per approval.  One screened
approval earns the same reward and consumes
\begin{equation}
a_i^S=\left(\frac{1}{\mu_{g,i}m_it_i},\,
\frac{1}{\mu_{r,i}m_it_i},\,
\frac{1}{\mu_{h,i}q_{\mathrm{acc},i}}\right).
\label{eq:activity_screened_vector}
\end{equation}
Screening requires \(1/(m_it_i)\) attempts per approval, because false
rejections regenerate merge-ready changes, and each attempt is reviewed.  Its
benefit is in the last coordinate.  Maintainer time per approval falls from
\(1/(\mu_{h,i}m_i)\) to \(1/(\mu_{h,i}q_{\mathrm{acc},i})\), which is a
reduction exactly when the reviewer is informative.

\paragraph{Capacity-allocation problem.}
The equivalent planning problem in approved flows is
\begin{subequations}\label{eq:activity_lp}
\begin{align}
V(n_h):=\max_{f^D,f^S\ge0}\quad&
\sum_i r_i(f_i^D+f_i^S)\label{eq:activity_obj}\\
\text{s.t.}\quad&
\sum_i\left(\frac{f_i^D}{\mu_{g,i}m_i}
+\frac{f_i^S}{\mu_{g,i}m_it_i}\right)\le n_g,\label{eq:activity_worker}\\
&\sum_i\frac{f_i^S}{\mu_{r,i}m_it_i}\le n_r,\label{eq:activity_judge}\\
&\sum_i\left(\frac{f_i^D}{\mu_{h,i}m_i}
+\frac{f_i^S}{\mu_{h,i}q_{\mathrm{acc},i}}\right)\le n_h,\label{eq:activity_expert}\\
&f_i^D+f_i^S\le\lambda_i,\qquad i\in\mathcal I.\label{eq:activity_arrival}
\end{align}
\end{subequations}
The first three constraints are the generation, AI-review, and
human-approval capacities; the last prevents approvals from exceeding
arrivals.  We write \(V(n_h)\) because maintainer capacity is the staffing
lever we study most closely in what follows.

\begin{lemma}[Reduction to approved flows]
\label{prop:approved_flow_reduction}
For every feasible point of \eqref{eq:full_fluid_program}, the flows
\(f_i^D=\mu_{h,i}m_i\hat z_{i,d}\) and
\(f_i^S=\mu_{h,i}q_{\mathrm{acc},i}\hat z_{i,s}\) are feasible for
\eqref{eq:activity_lp} and earn the same reward.  Conversely, every feasible
\((f^D,f^S)\) in \eqref{eq:activity_lp} is realized by a feasible point of
the full program with empty review and approval queues.  Hence
\(V_{\mathrm{full}}(n_h)=V(n_h)\).
\end{lemma}

The vectors \(a_i^D\) and \(a_i^S\) already include the expected capacity
per approved change; excess arrivals wait in the generation queue.  The
equivalence concerns maximum throughput only; downstream obsolescence
reappears in the stability analysis of
Section~\ref{sec:stoch_control}.

\section{Optimal AI-Review Allocation}
\label{sec:canonical}

Problem~\eqref{eq:activity_lp} is a linear program, so its solution is
governed by which capacities bind.  We first solve it for a single category,
where the only decision is how much work to review, and then characterize
how the allocation across heterogeneous categories responds to maintainer
staffing.

\subsection{Single Class: How Much to Review}
\label{subsec:single_class}

With one class we drop the class index; two conditions describe the regime.

\begin{assumption}[Generation overload]
\label{ass:overloaded_single}
Arrivals exceed capacity: \(\lambda\ge\mu_g(1-\alpha)n_g\).
\end{assumption}

\begin{assumption}[Informative review]
\label{ass:informative_single}
Review is informative and imperfect:
\(q_{\mathrm{acc}}>1-\alpha\), \(\beta^{(I)}>0\).
\end{assumption}

Assumption~\ref{ass:overloaded_single} makes the arrival
constraint~\eqref{eq:activity_arrival} redundant.
Assumption~\ref{ass:informative_single} makes review worth considering, at
the cost of a fraction \(\beta^{(I)}\) of merge-ready work per reviewed
unit; an uninformative reviewer is bypassed and leaves nothing to decide.

\paragraph{A two-variable form.}
Measure AI-review and human capacity in generation-equivalent units,
\begin{equation}
E := \frac{\mu_h}{\mu_g}\,n_h,
\qquad
J := \frac{\mu_r}{\mu_g}\,n_r,
\qquad
J_{\mathrm{eff}} := \min\{n_g,\,J\}.
\label{eq:sc_normalized_capacities}
\end{equation}
Let \(x\) be the generation capacity in use and \(v\le x\) the part whose
output is sent to AI review, so that the review fraction is
\(\phi=v/x\).  Dividing the objective of \eqref{eq:activity_lp} by
\(r\mu_g(1-\alpha)\) yields the reduced program
\begin{equation}
\max_{x,v}\ x-\beta^{(I)}v
\quad\text{s.t.}\quad
0\le v\le x\le n_g,\quad v\le J,\quad x-p_{\mathrm{rej}}v\le E.
\label{eq:sc_lp}
\end{equation}
The objective is generation in use net of the merge-ready work that review
sends back.  The three upper bounds are generation, AI review, and human
approval; the last one shows how review relieves maintainers, since each
reviewed unit removes \(p_{\mathrm{rej}}\) units of human load.

\paragraph{Thresholds.}
Three maintainer-capacity levels separate the operating phases:
\begin{equation}
n_h^{(1)} := \frac{\mu_g}{\mu_h}\,p_{\mathrm{pass}}\,J_{\mathrm{eff}},
\quad
n_h^{(2)} := \frac{\mu_g}{\mu_h}\,\big(n_g - p_{\mathrm{rej}}\,J_{\mathrm{eff}}\big),
\quad
n_h^{(3)} := \frac{\mu_g}{\mu_h}\,n_g.
\label{eq:sc_thresholds_123}
\end{equation}
They satisfy \(n_h^{(1)}\le n_h^{(2)}\le n_h^{(3)}\), strictly when
\(0<J_{\mathrm{eff}}<n_g\).  The first is the maintainer capacity that
absorbs all screened output when everything is reviewed; the second is where
generation saturates; the third absorbs all generated output without review.

\begin{proposition}[Four-phase structure]
\label{prop:four_phase}
Under Assumptions~\ref{ass:overloaded_single}
and~\ref{ass:informative_single}, suppose \(0<J_{\mathrm{eff}}<n_g\) and
\(n_h>0\).  The optimal review fraction \(\phi^*=v^*/x^*\) is
\begin{enumerate}
\item \(\phi^*=1\) if \(n_h\le n_h^{(1)}\) \emph{(full review)};
\item \(\phi^*=J_{\mathrm{eff}}/(E+p_{\mathrm{rej}}J_{\mathrm{eff}})\)
if \(n_h^{(1)}\le n_h\le n_h^{(2)}\) \emph{(reviewer saturation)};
\item \(\phi^*=(n_g-E)/(p_{\mathrm{rej}}n_g)\)
if \(n_h^{(2)}\le n_h\le n_h^{(3)}\) \emph{(active throttling)};
\item \(\phi^*=0\) if \(n_h\ge n_h^{(3)}\) \emph{(bypass)}.
\end{enumerate}
\end{proposition}

\begin{figure}[htbp]
\centering
\includegraphics[width=0.84\textwidth]{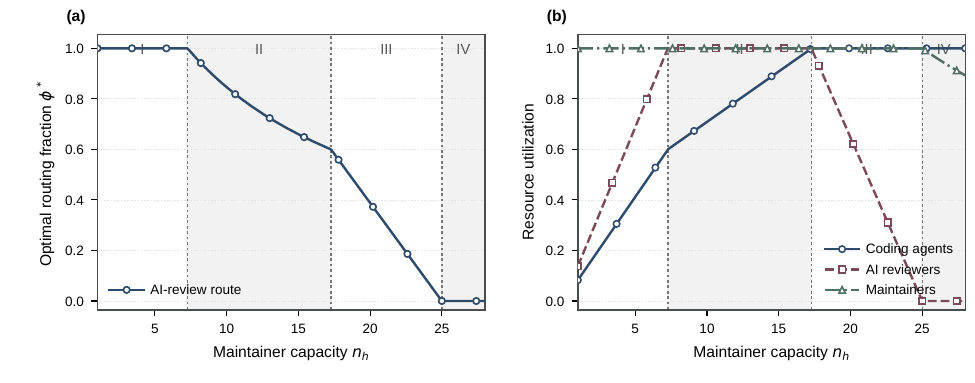}
\caption{Four-phase structure under the calibrated medium-quality error
profile.  Panel (a) shows the optimal review fraction; Panel (b) shows which
resource binds.  Vertical guides are the thresholds
in~\eqref{eq:sc_thresholds_123}.}
\label{fig:calibrated_phases}
\end{figure}

Each phase corresponds to a different binding resource
(Figure~\ref{fig:calibrated_phases}).  In Phases~1 and~2 maintainers are
the bottleneck and review protects them: with little maintainer capacity
every generated change is reviewed, because review converts one unit of
output into only \(p_{\mathrm{pass}}<1\) units of approval work, and once
the reviewer is saturated (\(v^*=J_{\mathrm{eff}}\)) additional maintainer
capacity is filled through the direct route, so the falling review fraction
is a change in composition rather than a decision to send less work to
review.  Phase~3 is different in kind.  Generation is now the bottleneck
(\(x^*=n_g\)) and every false rejection displaces a new task, so the plan
cuts reviewed volume itself, \(v^*=(n_g-E)/p_{\mathrm{rej}}\), leaving an
informative reviewer idle to protect generation.  In Phase~4 maintainers
absorb all generated output directly and review is bypassed.  The width of
Phase~3, \(n_h^{(3)}-n_h^{(2)}=(\mu_g/\mu_h)p_{\mathrm{rej}}J_{\mathrm{eff}}\),
grows with reviewer capacity, because a larger reviewer can create more
rework for the plan to throttle; Phase~2 vanishes as \(J_{\mathrm{eff}}\)
approaches \(n_g\).

\subsection{Multiple Classes: Which Categories to Review}
\label{subsec:multiclass}

With several categories the operator must also decide which ones receive
scarce review capacity.  Post-review quality alone is an unreliable guide,
because a more accurate reviewer may cost more regeneration, reviewer time,
or maintainer time per approved change; duality compares reward net of
these three costs at their current prices.

\paragraph{Resource prices.}
Let \(\omega,\kappa,\chi\ge0\) be the prices of generation, AI-review, and
human-approval capacity, and let \(\nu_i\ge0\) be the price of class
\(i\)'s arrival cap.  Writing \(a_{g,i}^m,a_{r,i}^m,a_{h,i}^m\) for the
coordinates of \(a_i^m\), the dual inequalities for the two routes of
class \(i\) read
\begin{align}
\omega a_{g,i}^D+\chi a_{h,i}^D+\nu_i&\ge r_i,
\label{eq:dual_direct}\\
\omega a_{g,i}^S+\kappa a_{r,i}^S+\chi a_{h,i}^S+\nu_i&\ge r_i,
\label{eq:dual_screened}
\end{align}
with equality on any route that carries flow.  Rearranging
\eqref{eq:dual_screened} defines the \emph{reviewer index}
\begin{equation}
\mathcal I_i(\omega,\chi,\nu_i)=
\frac{r_i-\nu_i-\omega a_{g,i}^S-\chi a_{h,i}^S}{a_{r,i}^S},
\label{eq:general_priority_index}
\end{equation}
the surplus a screened class-\(i\) approval earns after paying for
generation, maintainer time, and its arrival cap, per unit of reviewer time
it uses.  A class is reviewed only if its index attains the reviewer price
\(\kappa\); classes with a lower index are not reviewed.

\begin{proposition}[Sparse multiclass staffing response]\label{thm:multiclass_path}
Fix \(n_g,n_r>0\), finite positive arrival rates, and all primitives other
than \(s:=n_h\).  There is a finite partition
\(0=s_0<s_1<\cdots<s_K<s_{K+1}=\infty\) such that:
\begin{enumerate}
\item \(V(s)\) is nondecreasing, concave, and piecewise affine with
nonincreasing slopes \(\chi_k\ge0\).
\item On each \((s_k,s_{k+1})\), an optimal extreme allocation has the form
\(f^{(k)}(s)=u_k+s\,d_k\) with a fixed set of active routes and a fixed
choice of optimal resource prices.
For every \(s\), some optimal extreme point obeys the sparsity bound
\begin{equation}
U(s)+L(s)\le3,
\label{eq:multiclass_sparsity}
\end{equation}
where \(U(s)\) counts active classes below their arrival caps and \(L(s)\)
counts active classes using both routes.
\item For small \(s>0\), positive flow uses only routes in
\(\arg\max_{i,\;m\in\{D,S\}} r_i/a_{h,i}^m\).  For large \(s\), the
optimum can be taken to be direct-only.
\end{enumerate}
\end{proposition}

The proposition turns the staffing question into a finite operating map of
breakpoints \(s_k\), active routes, and affine rules \(u_k+s\,d_k\); the
slope \(\chi_k\) is the throughput value of one more maintainer.  The bound
\(U(s)+L(s)\le3\) counts the categories whose allocation is still
adjustable: capped categories need no reallocation and inactive ones
receive nothing, so at most three categories determine the response to a
staffing change, the fact exploited in Section~\ref{sec:stoch_control}.
Figure~\ref{fig:staffing_map}(a) shows such a map.

Proposition~\ref{thm:multiclass_path} says what the map looks like, not
which categories it reviews.  A natural guess is that this part never
changes: rank the categories once by how good their reviewers are, and let
staffing decide only how much of the top category to review.  The next
result checks this guess in the case most favorable to it, where the
categories are identical except for their reviewers.  It needs one more
quantity, \(\eta_i:=\beta_i^{(I)}/p_{\mathrm{rej},i}\), the merge-ready
work that reviewing class \(i\) sends back for each unit of maintainer
work it saves.  We call \(q_{\mathrm{acc},i}\) the \emph{filter quality}
and \(\eta_i\) the \emph{rework cost} of reviewer \(i\); \(E\) and
\(J_{\mathrm{eff}}\) are as in \eqref{eq:sc_normalized_capacities}.

\begin{proposition}[Priority switch]
\label{prop:reviewed_switch}
Let all classes satisfy Assumptions~\ref{ass:overloaded_single}
and~\ref{ass:informative_single}, share \(r,\mu_g,\mu_r,\mu_h,\alpha\), and
differ only in \((\beta_i^{(I)},\beta_i^{(II)})\).  Let \(i_A\) uniquely
maximize \(q_{\mathrm{acc},i}\) and \(i_B\) uniquely minimize \(\eta_i\).
\begin{enumerate}
\item If \(E\le p_{\mathrm{pass},i_A}J_{\mathrm{eff}}\), the unique optimum
reviews all generated class-\(i_A\) work and serves nothing else.
\item If \(n_g-p_{\mathrm{rej},i_B}J_{\mathrm{eff}}\le E<n_g\), every
optimum uses all generation capacity and reviews only class \(i_B\), with
reviewed volume \((n_g-E)/p_{\mathrm{rej},i_B}\).
\end{enumerate}
The first interval lies below the second.  If \(i_A\ne i_B\), the reviewed
category switches as maintainer capacity grows although no reviewer changes.
\end{proposition}

Even here no single ranking works, because the two ends of the map care
about different mistakes.  With few maintainers, every approval must earn
as much as possible per maintainer hour, so the best reviewer is the one
that lets the fewest defective changes through for each merge-ready change
it passes: \(q_{\mathrm{acc}}\) ranks by false acceptance.  Once generation
is the bottleneck, every review is paid for with regeneration, so the best
reviewer is the one that sends back the fewest merge-ready changes for each
defect it catches: \(\eta\) ranks by false rejection.
Figure~\ref{fig:staffing_map} is that statement for three reviewers whose
error rates sit near the calibrated tiers of
Section~\ref{subsec:calibration}.  In panel (b), L is \(i_A\) and the two
lines are its rankings: the solid line is who matches it on
\(q_{\mathrm{acc}}\), the dashed line who matches it on \(\eta\).  H lies
in the wedge, so it loses the first ranking and wins the second
(\(i_B=\mathrm{H}\)); M loses both and is never reviewed.  The arrow is
the middle interval of panel (a).  As maintainers are added the shadow
price of approval falls and that of generation rises, so the relevant
ranking rotates from \(q_{\mathrm{acc}}\) toward \(\eta\) and review
moves from \(i_A\) to \(i_B\).  The two ends of panel (a) are the
intervals in Proposition~\ref{prop:reviewed_switch}; between them both
are reviewed.  One accuracy
number cannot capture this; an operator must track both error rates and
let the bottleneck decide which one matters.  With unequal rewards and
service rates the prices in \eqref{eq:general_priority_index} make the same
comparison, and Section~\ref{subsec:exp_calibrated_which} measures the cost
of ranking once by either statistic.

\begin{figure}[htbp]
\centering
\includegraphics[width=0.96\textwidth]{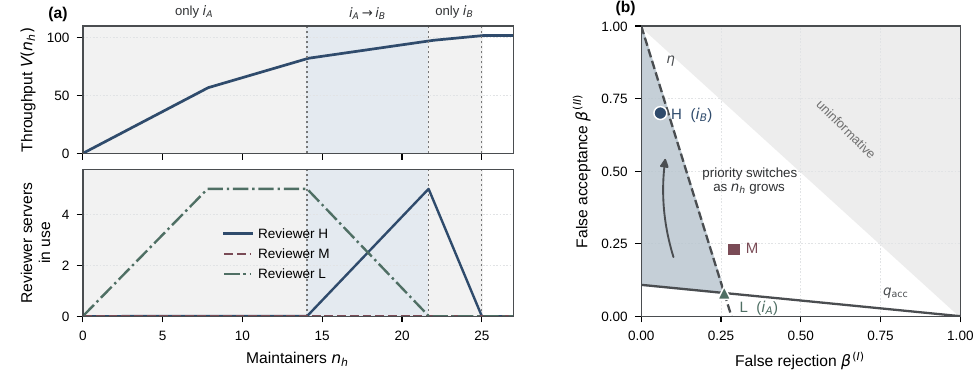}
\caption{Which reviewer to use depends on the bottleneck (illustrative
instance near the calibrated error rates; categories identical except for
reviewers H, M, L).
(a) Only \(i_A\), then both as the priority switches, then only \(i_B\).
(b) Why \(i_A\ne i_B\): H is a worse filter than L but cheaper in rework.
The arrow is the priority switch in (a).}
\label{fig:staffing_map}
\end{figure}

\section{Quota Tracking}
\label{sec:stoch_control}

The staffing map prescribes average flows.  The operating system must
decide at every completion and at every server that becomes free, starting
from whatever state it is in.  Three properties of the map from
Section~\ref{sec:canonical} shape the policy, and each motivates one design
choice.

First, by Propositions~\ref{thm:multiclass_path}
and~\ref{prop:reviewed_switch} the optimal flows are piecewise affine in
staffing and the reviewed category can switch at a breakpoint, so the policy
stores the whole map and reads the current interval, rather than a fixed
ranking.  Second, the plan is fractional.  In Phases~2
and~3 of Proposition~\ref{prop:four_phase} the review fraction lies strictly
between zero and one, and by \eqref{eq:multiclass_sparsity} at most three
categories are split across routes.  The simplest way to realize a fraction
\(\phi_i^*\) is randomized routing: solve the LP, then send each generation
completion to review independently with probability \(\phi_i^*\).  This
attains the planning bound in the fluid limit just as the policy below does,
but after \(k\) completions the realized count deviates from \(\phi_i^*k\)
by order \(\sqrt k\), and each excess review feeds false rejections back
into the generation queue.  The policy instead keeps a cumulative quota for
each split category, so the deviation never exceeds one and at most three
counters are ever live; this is a finite-system refinement rather than a
first-order gain, and Theorem~\ref{thm:asymptotic_optimality} does not
distinguish the two.  Third, in Phase~3 the plan holds informative review
capacity idle, and its server shares are fractional.  Any work-conserving
rule fills the idle share and recreates the rework that
Proposition~\ref{prop:four_phase} shows to be harmful, so idleness must be
planned; fixed integer subpools would plan it but lose capacity to rounding,
so the policy pools each server group and gives planned idleness its own
quota.  Every planned
fraction thus becomes a running count of what the plan has earned so far,
and work is admitted so that realized counts stay within a bounded distance
of these quotas; hence the name Quota Tracking.

\subsection{Policy}
\label{subsec:policy-design}

\paragraph{Offline staffing map.}
Compute the breakpoints, active routes, prices, and affine rules of
Proposition~\ref{thm:multiclass_path}.  At the observed maintainer capacity
\(s=n_h\), look up its interval and set \(f^*=u_k+s\,d_k\); at a tie, select
one optimal extreme point by a fixed rule and keep it until the inputs
change.  Convert the flows into generation targets and a review fraction,
\begin{equation}
v_i^*=\frac{f_i^{S,*}}{\mu_{g,i}m_it_i},\qquad
x_i^*=v_i^*+\frac{f_i^{D,*}}{\mu_{g,i}m_i},\qquad
\phi_i^*=\begin{cases}v_i^*/x_i^*,&x_i^*>0,\\0,&x_i^*=0,\end{cases}
\label{eq:phi_star}
\end{equation}
and into the servers the plan keeps busy, per unit scale, at each node:
\begin{equation}
y_i^*=\frac{\mu_{g,i}}{\mu_{r,i}}v_i^*,\qquad
z_{i,d}^*=\frac{\mu_{g,i}}{\mu_{h,i}}(x_i^*-v_i^*),\qquad
z_{i,s}^*=\frac{\mu_{g,i}}{\mu_{h,i}}p_{\mathrm{pass},i}v_i^*.
\label{eq:service_targets}
\end{equation}
Together with \(x_i^*\), these are the service targets, and \(b_c^*\) is
the target of a generic pool coordinate \(c\), meaning one class at one node.

\paragraph{Deterministic route quotas.}
Let \(G_i^n\) count class-\(i\) generation completions already routed and
\(H_i^n\) those sent to AI review.  Send the next completion to review
exactly when \(H_i^n<\lfloor\phi_i^*(G_i^n+1)\rfloor\).  Then
\(|H_i^n-\phi_i^*G_i^n|<1\) after every decision.  Categories with
\(\phi_i^*\in\{0,1\}\) need no counter, so at most three counters are live.

\paragraph{Stabilize, then pool.}
From a cold start, round the service targets into capacity-feasible integer
reservations and refill each reservation whenever it has an open slot and
waiting work.  After each event, each pool makes at most one additional
catch-up admission to its largest reservation shortfall, so initialization
never creates a fluid-scale batch move; a coordinate above its reservation
receives no new admission until it drains.  Fix \(\delta_n\downarrow0\)
with \(n\delta_n\to\infty\).  Release the reservations when every service
coordinate is within \(n\delta_n\) of its target, and return to
stabilization only if some coordinate later deviates by more than
\(2n\delta_n\).  In pooled mode, for each coordinate \(c\) with \(X_c^n(t)\)
servers in use, track the cumulative effort deficit
\begin{equation}
D_c^n(t)=nb_c^*t-\int_0^tX_c^n(s)\,ds,
\label{eq:effort_deficit}
\end{equation}
reset to zero when pooled mode begins, and include an idle coordinate whose
target is the pool's planned unused capacity.  When a server becomes free,
assign it to the nonempty coordinate, or to idle, with the largest deficit
\(D_c^n(t)\), breaking ties by the largest shortfall \(nb_c^*-X_c^n(t)\);
a server assigned to idle is reassigned by the same rule when work arrives
to its pool.  Pooled mode thus time-shares fractional targets without
partitioning servers, and the idle coordinate enforces the Phase~3
prescription: a free reviewer goes idle rather than to waiting work
whenever the plan has already earned that idle share.

\begin{remark}[Computational cost]\label{rem:complexity}
Let \(I=|\mathcal I|\) and let \(K\) be the number of breakpoints in
Proposition~\ref{thm:multiclass_path}.  Offline, the map is a parametric
linear program in \(2I\) variables and \(I+3\) constraints, one basis per
staffing interval.  Online at a fixed staffing level no optimization is
solved: each generation completion updates one counter, at most three
counters are live, and each freed server compares the deficits in its pool,
\(O(I)\) arithmetic.  Re-staffing adds an interval lookup, \(O(\log K)\),
and a recomputation of \eqref{eq:phi_star}--\eqref{eq:service_targets},
\(O(I)\); within an interval only the targets of at most three categories
move, so stabilization restarts with every other reservation at target.  The
map itself is recomputed only when the primitives change.
\end{remark}

\subsection{Performance Guarantee}
\label{subsec:asymptotic_opt}

The guarantee takes the system scale to infinity on each fixed horizon and
then the horizon to infinity.  Let \(\mathcal X^*\) denote the stationary
state that realizes the selected optimum: service coordinates at the
targets \eqref{eq:phi_star}--\eqref{eq:service_targets}, empty review and
approval queues, and a generation queue holding the arrivals the plan does
not approve until they become obsolete (empty for a capped category).

\begin{theorem}[Cold-start convergence and asymptotic optimality]
\label{thm:asymptotic_optimality}
Select an optimal extreme point of \eqref{eq:activity_lp} by the fixed
tie-breaking rule and let \(R^*=V(n_h)\).  Under
Assumptions~\ref{ass:regular_params} and~\ref{ass:fluid-init}, let
\(\boldsymbol{\pi}^{n,*}\) denote
Quota Tracking started from any initial-state sequence satisfying
Assumption~\ref{ass:fluid-init}.  Every fluid limit of the queue and
service state converges to \(\mathcal X^*\), and
\begin{equation}
\label{eq:asymptotic_opt}
\liminf_{T\to\infty}\;\liminf_{n\to\infty}\; \frac{1}{nT}\,\mathbb{E}^{\boldsymbol{\pi}^{n,*}}\Big[\sum_{i\in\mathcal{I}} r_i\, C_i^n(T)\Big] \;=\; R^*.
\end{equation}
\end{theorem}

The inner limit applies the fluid approximation on a fixed horizon; the
outer limit lets the cold-start transient vanish.  This fluid-first statement
shows that the planning bound is attainable; optimality for
\eqref{eq:pf_obj} at each fixed \(n\) would need the reverse order and
remains open.  The proof, sketched in
Appendix~\ref{app:sketch_asymptotic_opt} and given in full in the Online
Supplement, shows that no admissible policy can exceed \(R^*\) under this
limit order, that the reservation rule drives every service coordinate to
its target from any bounded fluid state, and that the quotas and the
largest-deficit rule then hold the time-averaged service vector there.  No
overload condition is needed: a category served up to its arrival cap has
no planned backlog, but at the target its generation coordinate is supplied
at exactly its service rate, like every review and approval coordinate, so
a shortfall appears as a backlog of the same size and is cleared.  A
breakpoint tie does not invalidate the guarantee, because the policy tracks
one selected optimum.  The theorem covers a fixed staffing level; when
capacity moves, the policy re-enters stabilization and the theorem applies
to the new target.

\section{Numerical Experiments}
\label{sec:numerical}

The analysis identifies three quantities an operator should measure.  The
first is the value of adapting review to the bottleneck, the second is the
cost of ranking categories by a single index, and the third is the
sensitivity of both decisions to estimated reviewer errors.  We design one
experiment for each and a fourth that measures how close Quota Tracking
comes to the fluid optimum in systems of realistic size.  Public data
calibrate code-change quality and reviewer errors; arrivals, service rates,
rewards, and installed capacities are managerial scenarios rather than
empirical estimates.

\subsection{Calibration}
\label{subsec:calibration}

SWE-Review-Bench reports review outcomes for 1,384 candidate pull
requests generated by three generator models on SWE-bench Verified issues,
with 500, 462, and 422 instances per tier
\citep{wang2026swereview,jimenez2024swebench}.  Using the Claude Opus 4.6
error counts in its Table~8 and the per-tier instance counts of the
benchmark release (the Online Supplement records the counts and an
alternative base), we map the unresolved share to
\(\alpha_i\), false rejection conditional on resolved candidates to
\(\beta_i^{(I)}\), and false acceptance conditional on unresolved candidates
to \(\beta_i^{(II)}\).  Figure~\ref{fig:empirical_calibration} gives point
estimates and 95\% Wilson intervals.  The estimates
\((\alpha_i,\beta_i^{(I)},\beta_i^{(II)})\) are \((0.278,0.058,0.727)\),
\((0.491,0.230,0.189)\), and \((0.725,0.259,0.075)\) for the high-,
medium-, and low-quality tiers.  We use the tiers as three error profiles,
not as observed operational task categories.  The calibration targets
functional issue resolution, the outcome SWE-bench measures; maintainability,
security, and style criteria are outside these aggregates.

\begin{figure}[htbp]
\centering
\includegraphics[width=0.96\textwidth]{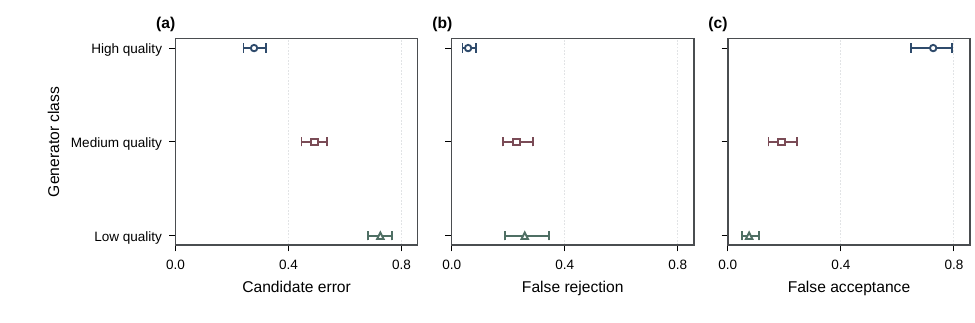}
\caption{Estimated code-change error and AI-reviewer error probabilities by
generator tier.  Intervals are 95\% Wilson intervals computed from published
aggregate counts.}
\label{fig:empirical_calibration}
\end{figure}

\subsection{Value of Adaptive Review}
\label{subsec:exp_calibrated_how_much}

\paragraph{Question and method.}
Does adapting the review fraction to both capacities beat a fixed
``review everything'' or ``review nothing'' rule, and which reviewer error
should an engineering team reduce first?  We use the medium-quality
profile, \((\mu_g,\mu_r,\mu_h)=(20,30,8)\) completions per server per day,
\(n_g=10\), and \(\lambda=150\) arrivals per day.  Over a grid of AI-review and maintainer capacities we
compare the optimal plan with the better of the two fixed rules.  We then set
the Type~I or the Type~II error to zero, one at a time, and record the
resulting throughput gain.

\begin{figure}[htbp]
\centering
\includegraphics[width=\textwidth]{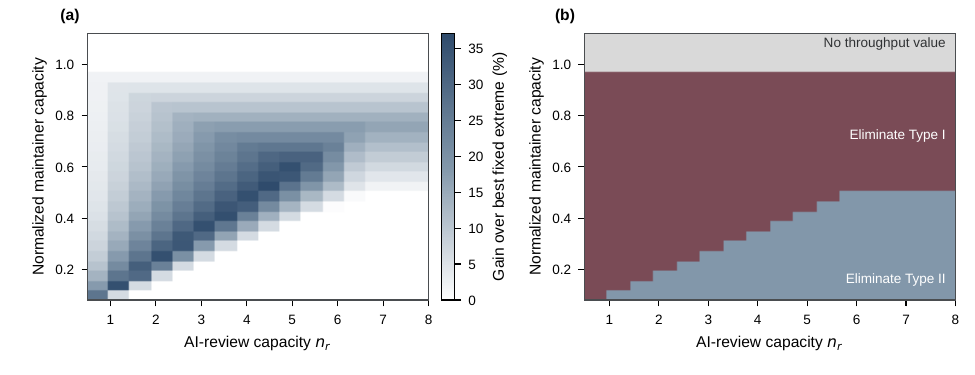}
\caption{Value of adaptive review.  Panel (a) reports percentage throughput
gain over the better of Always-Review and Never-Review.  Panel (b)
identifies which perfect-error counterfactual, eliminating false rejection
(Type~I) or false acceptance (Type~II), has greater value at each capacity
pair.}
\label{fig:adaptive_value}
\end{figure}

\paragraph{Finding and interpretation.}
Adaptive routing strictly improves on both fixed rules on 62.7\% of the
capacity grid, with a maximum gain of 37.0\% at an interior configuration
where every resource is scarce.  The error that matters most changes with
the bottleneck.  When maintainers are scarce and review capacity is ample,
eliminating false acceptance is more valuable, because it keeps defective
work away from the human gate.  When review capacity itself is scarce,
eliminating false rejection is more valuable, because each wasted review
slot is costly.  From \(n_h=13\) through \(24\), eliminating false
rejection is more valuable at every review-capacity level, since generation
has become the scarce input; from \(n_h=25\) review is bypassed and neither
improvement changes throughput.  The value of reviewer engineering therefore
cannot be read from an aggregate accuracy score: the same model should be
tuned for precision when maintainer attention binds and for recall when
coding-agent capacity binds.  Proposition~\ref{prop:four_phase} predicts
where the switch occurs, and the experiment measures how large that
switch is.

\subsection{Value of Multiclass Allocation}
\label{subsec:exp_calibrated_which}

\paragraph{Question and method.}
How costly is a single-index category ranking?  The three error profiles
are combined with the managerial scenario
\[
\mu_g=(16,20,24),\quad \mu_r=(24,30,34),\quad
\mu_h=(7,8,9),\quad r=(1,1,1),
\]
\(n_g=10\), \(n_r=3\), and nonbinding arrival caps, so that AI review can
cover about half of generation output.  Rewards are equal, so the profiles
differ only in quality and in service requirements.  Two single-index rules are natural.
The \(q_{\mathrm{acc}}\)-only rule reviews only the profile with the
highest post-review quality.  The \(\eta\)-only rule reviews only the
profile with the smallest false-rejection burden \(\eta_i\); by
Proposition~\ref{prop:reviewed_switch}, each rule is optimal at one end of
the staffing range when profiles differ only in their reviewers.  We
compare the optimal plan with these two rules and
with never and always reviewing.  Three counterfactuals equalize quality,
equalize service requirements, or bind one arrival cap.

\begin{figure}[htbp]
\centering
\includegraphics[width=\textwidth]{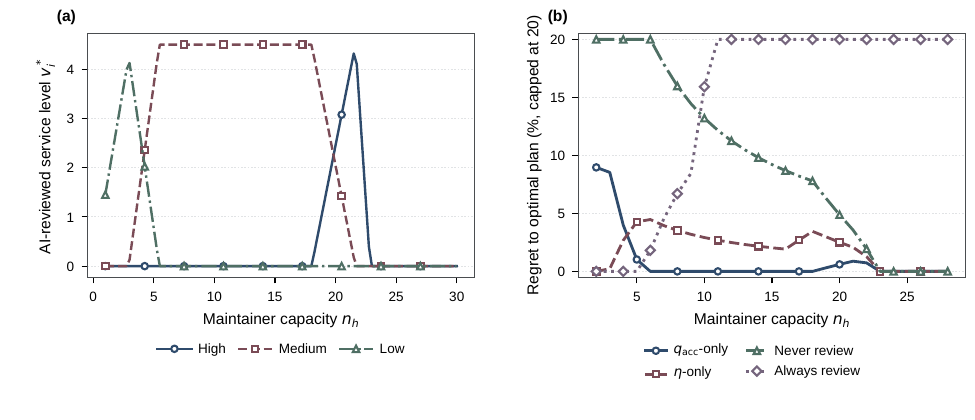}
\caption{Multiclass review decisions in the baseline scenario.  Panel (a)
shows the AI-reviewed service level of each profile as maintainer staffing
changes.  Panel (b) reports the regret of restricted routing rules relative
to the optimal plan.  The Never-Review and Always-Review regrets reach
28.8\% and 57.6\% and are capped at 20\% in the panel.}
\label{fig:multiclass_decisions}
\end{figure}

\paragraph{Finding and interpretation.}
The reviewed profile moves up the quality ladder as maintainers are added.
With very few maintainers only the low-quality profile is generated and
reviewed, because heavy filtering yields the most approvals per maintainer
hour.  The medium-quality profile joins from \(n_h=3\) and is reviewed alone
from \(n_h\approx6\); the high-quality profile is added with a split route
from \(n_h\approx18\) and reviewed alone from \(n_h\approx22\); from
\(n_h=23\) review is bypassed.  At most two profiles are reviewed and one is
split, as Proposition~\ref{thm:multiclass_path} allows.  Neither single index
anticipates this sequence.  The \(q_{\mathrm{acc}}\)-only rule always picks the
medium-quality profile and loses 9.0\% when maintainers are scarce; the
\(\eta\)-only rule always picks the low-quality profile and loses up to
4.5\% once maintainers are moderately staffed.  Never-Review loses up to
28.8\% when maintainers are scarce.  Always-Review loses up to 57.6\% when
they are not, because with three review servers it caps generation at what
the reviewer can inspect.  No fixed rule is good at every staffing level,
and the plan tells the operator when to switch.

The counterfactuals locate the source of the switches.  With equalized
quality, service requirements alone select the fastest profile for review
at every staffing level.  With equalized service requirements, the reviewed
profile still changes from medium to high as staffing grows, so quality
differences alone reorder priority.  When an arrival cap binds, all three
profiles are reviewed at once for \(17\le n_h\le20\), because the capped
profile cannot absorb the flow that would otherwise replace the others.  The operator should therefore update
the small set of marginal profiles at each breakpoint, using the generation
and maintainer prices, rather than rank categories once by a quality score.

\subsection{Calibration Uncertainty}
\label{subsec:exp_uncertainty}

\paragraph{Question and method.}
The published evidence is a set of aggregate confusion counts.  We place a
Jeffreys Dirichlet prior on each tier's four outcome cells, propagate 10,000
posterior draws through the plan, and report the posterior probability that
each profile receives review, bands for the single-class thresholds, and the
regret from applying the plug-in routing fractions when a posterior draw is
the true model.

\begin{figure}[htbp]
\centering
\includegraphics[width=\textwidth]{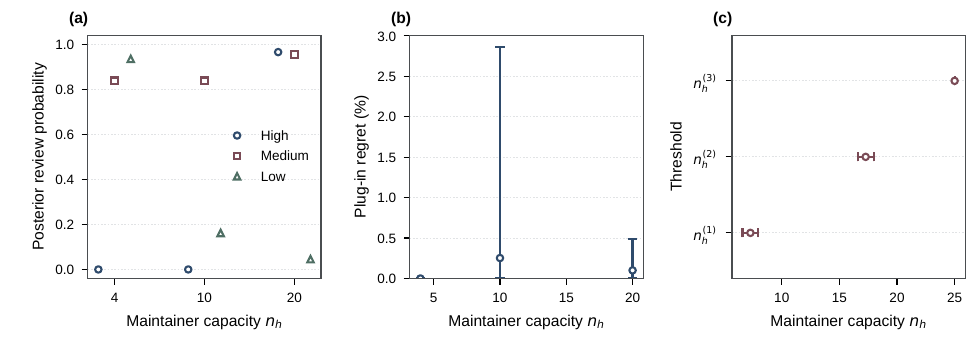}
\caption{Decision uncertainty from aggregate review counts.  Panel (a)
reports posterior probabilities that each profile receives AI review at
three maintainer capacities.  Panel (b) reports the distribution of plug-in
policy regret.  Panel (c) reports posterior bands for the single-class
staffing thresholds.}
\label{fig:calibration_uncertainty}
\end{figure}

\paragraph{Finding and interpretation.}
Uncertainty concentrates where the reviewed profiles change.  At
\(n_h=4\) the low- and medium-quality profiles are reviewed in 93.5\% and
83.8\% of posterior draws and the high-quality profile never is.  At
\(n_h=10\), where the plug-in plan reviews only the medium-quality profile,
the low-quality profile still enters in 16.2\% of draws.  At \(n_h=20\) the
high- and medium-quality profiles are reviewed in 96.5\% and 95.4\% of
draws.  The 95\% bands for the first two single-class thresholds are
\([6.59,7.97]\) and \([16.59,17.97]\); the bypass threshold is fixed at 25
by service capacity.  Plug-in routing has zero regret at \(n_h=4\), zero
median regret at \(n_h=10\) with a 97.5th-percentile loss of 2.9\%, and a
97.5th-percentile loss of 0.5\% at \(n_h=20\).  Holding service parameters and the error model fixed, point
estimates are adequate away from a breakpoint and less reliable near one, so
a deployment should pair each planned breakpoint with a posterior review
probability or a regret band.

\subsection{Finite-System Performance}
\label{subsec:exp_finite}

\paragraph{Question and method.}
How much of the fluid optimum does Quota Tracking deliver in a system of
realistic size, and how does its gap compare with the rules an operator
might use instead?  We simulate the stochastic network of
Section~\ref{subsec:exp_calibrated_which} at scales
\(n\in\{1,2,4,8,16\}\), from 10 coding-agent sessions, 3 review sessions,
and 8 to 20 maintainers up to 16 times as many of each.  The three
maintainer capacities are \(n_h\in\{8,12,20\}\), where the plan reviews
the medium-quality profile alone, alone, and together with a split
high-quality profile; every heuristic has positive fluid regret.  For each policy we
report the relative gap between its long-run reward rate and \(nV(n_h)\),
the fluid optimum at that scale.  The benchmarks are Always-Review,
Never-Review, and the \(\eta\)-only single-index rule, each with its own
best fluid plan.  Twenty
replications of length 150 with a warm-up of 30 use common random numbers;
error bars are 95\% confidence intervals.

\begin{figure}[htbp]
\centering
\includegraphics[width=\textwidth]{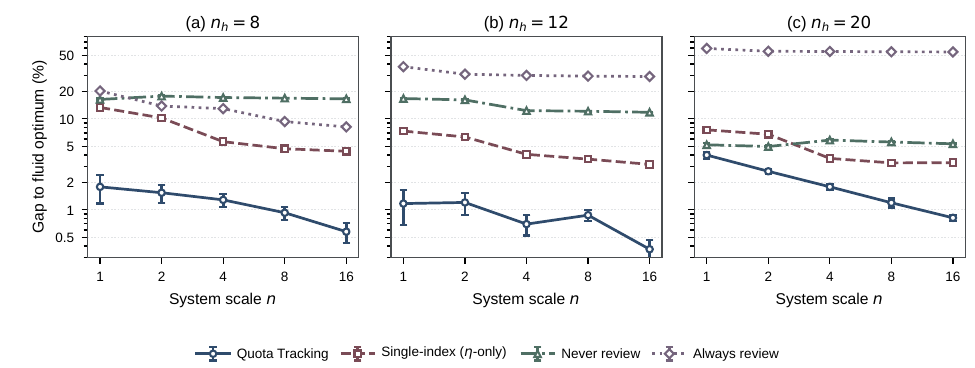}
\caption{Relative gap to the fluid optimum as a function of system scale
(log--log axes).  Each panel fixes maintainer capacity; each curve is a
policy.  Error bars are 95\% confidence intervals from twenty replications
that share common random numbers.}
\label{fig:finite_system_gap}
\end{figure}

\paragraph{Finding and interpretation.}
Figure~\ref{fig:finite_system_gap} shows two kinds of curves.  The gap of
Quota Tracking falls with scale, from about 1.8\%, 1.2\%, and 4.0\% at
\(n=1\) to about 0.6\%, 0.4\%, and 0.8\% at \(n=16\) for
\(n_h=8\), \(12\), and \(20\).  The gaps of the three heuristics do not fall.  Each
settles at the fluid regret of its own plan, about 5\% to 17\% for
Never-Review, 3\% to 4\% for the single-index rule, and 8\% to 54\% for
Always-Review, because the loss is built into the plan they follow rather
than into its execution.  The fixed rules ignore the interior review
fractions of Proposition~\ref{prop:four_phase} and the single-index rule
ignores the breakpoint switches of Proposition~\ref{thm:multiclass_path}, and
scale repairs neither.  The remaining gap of Quota Tracking at \(n=1\) is
the price of variability with about a dozen servers per pool, where
fractional shares must be time-shared and a loaded pool is sometimes idle
for lack of waiting work.  The planning bound of
Theorem~\ref{thm:asymptotic_optimality} is therefore not only a limit
statement: at the scale of a large engineering organization, Quota Tracking
leaves well under one percent of approved throughput on the table, whereas
the natural heuristics leave several percent regardless of size.

\section{Conclusion}
\label{sec:conclusion}

AI review before human approval is a capacity-allocation problem as much
as an accuracy problem.  Review filters defective changes before they consume
maintainer attention, but it consumes inference and, through false
rejection, coding-agent capacity; which effect dominates depends on the
bottleneck.  For a single category, the optimal review fraction passes
through four phases as maintainer capacity grows, and in one of them an
informative reviewer with idle capacity is deliberately throttled because
its rework displaces new work.  Across heterogeneous categories, the optimal
allocation changes only at finitely many staffing breakpoints, at most three
categories are marginal between them, and the reviewed category is chosen
by resource prices rather than by any single accuracy statistic.

These structural results yield an implementable policy.  Quota Tracking
stores the staffing map offline, routes by deterministic cumulative quotas,
and shares servers through pooled effort deficits.  It attains the planning
bound in the fluid limit, and in calibrated simulations its gap to that
bound shrinks with system size while fixed and single-index rules do not.
Operators should estimate both reviewer error types and the service
requirements of each category, compute the staffing breakpoints, and change
review intensity and category priority at those points.  Natural next steps
are extensions to reviewer accuracy that improves with feedback, to multiple
review stages, to staffing that responds to queue length, and to the joint
choice of generator model and review intensity under one inference budget.

\ifdefined\MainIncludesAppendix\else
\bibliography{references}
\fi

\bibliography{references}

@article{atar2010cmu,
  title={The $c\mu/\theta$ rule for many-server queues with abandonment},
  author={Atar, Rami and Giat, Chanit and Shimkin, Nahum},
  journal={Operations Research},
  volume={58},
  number={5},
  pages={1427--1439},
  year={2010},
  OPTdoi={10.1287/opre.1100.0826}
}

@article{ata2006dynamic,
  title={Dynamic Control of a Multiclass Queue with Thin Arrival Streams},
  author={Ata, Bar{\i}{\c{s}}},
  journal={Operations Research},
  volume={54},
  number={5},
  pages={876--892},
  year={2006},
  OPTdoi={10.1287/opre.1060.0308}
}

@article{dai1995positive,
  title={On positive {H}arris recurrence of multiclass queueing networks: a unified approach via fluid limit models},
  author={Dai, J. G.},
  journal={The Annals of Applied Probability},
  volume={5},
  number={1},
  pages={49--77},
  year={1995},
  doi={10.1214/aoap/1177004828}
}

@article{fugener2022cognitive,
  title={Cognitive Challenges in Human--Artificial Intelligence Collaboration: Investigating the Path Toward Productive Delegation},
  author={F{\"u}gener, Andreas and Grahl, J{\"o}rn and Gupta, Alok and Ketter, Wolfgang},
  journal={Information Systems Research},
  volume={33},
  number={2},
  pages={678--696},
  year={2022},
  OPTdoi={10.1287/isre.2021.1079}
}

@article{hu2022dynamic,
  title={Dynamic Type Matching},
  author={Hu, Ming and Zhou, Yun},
  journal={Manufacturing \& Service Operations Management},
  volume={24},
  number={1},
  pages={125--142},
  year={2022},
  OPTdoi={10.1287/msom.2020.0952}
}

@article{gurvich2015dynamic,
  title={On the dynamic control of matching queues},
  author={Gurvich, Itai and Ward, Amy},
  journal={Stochastic Systems},
  volume={4},
  number={2},
  pages={479--523},
  year={2014},
  OPTdoi={10.1287/13-ssy097}
}

@article{long2024generalized,
  title={The Generalized $c/\mu$ Rule for Queues with Heterogeneous Server Pools},
  author={Long, Zhenghua and Zhang, Hailun and Zhang, Jiheng and Zhang, Zhe George},
  journal={Operations Research},
  volume={72},
  number={6},
  pages={2488--2506},
  year={2024},
  OPTdoi={10.1287/opre.2023.2472}
}

@article{long2020dynamic,
  title={Dynamic Scheduling of Multiclass Many-server Queues with Abandonment: the Generalized $c\mu/h$ Rule},
  author={Long, Zhenghua and Shimkin, Nahum and Zhang, Hailun and Zhang, Jiheng},
  journal={Operations Research},
  volume={68},
  number={4},
  pages={1218--1230},
  year={2020},
  OPTdoi={10.1287/opre.2019.1908}
}

@article{whitt2006fluid,
  title={Fluid models for multiserver queues with abandonments},
  author={Whitt, Ward},
  journal={Operations Research},
  volume={54},
  number={1},
  pages={37--54},
  year={2006},
  OPTdoi={10.1287/opre.1050.0227}
}

@article{zhang2013fluid,
  title={Fluid Models of Many-server Queues with Abandonment},
  author={Zhang, Jiheng},
  journal={Queueing Systems},
  volume={73},
  number={2},
  pages={147--193},
  year={2013},
  OPTdoi={10.1007/s11134-012-9307-9}
}

@article{zhong2025optimal,
  title={Optimal Integration: Human, Machine, and Generative {AI}},
  author={Zhong, Hongda},
  journal={Management Science},
  volume={72},
  number={5},
  pages={3720--3739},
  year={2026},
  doi={10.1287/mnsc.2024.07401}
}

@inproceedings{zheng2023judging,
  title={Judging {LLM}-as-a-Judge with {MT-Bench} and {Chatbot Arena}},
  author={Zheng, Lianmin and Chiang, Wei-Lin and Sheng, Ying and Zhuang, Siyuan and Wu, Zhanghao and Zhuang, Yonghao and Lin, Zi and Li, Zhuohan and Li, Dacheng and Xing, Eric P. and Zhang, Hao and Gonzalez, Joseph E. and Stoica, Ion},
  booktitle={Advances in Neural Information Processing Systems (NeurIPS)},
  year={2023},
  note={Datasets and Benchmarks Track}
}

@inproceedings{zhu2023judgelm,
  title={{JudgeLM}: Fine-tuned Large Language Models are Scalable Judges},
  author={Zhu, Lianghui and Wang, Xinggang and Wang, Xinlong},
  booktitle={International Conference on Learning Representations (ICLR)},
  year={2025},
  note={Spotlight}
}

@inproceedings{tan2024judgebench,
  title={{JudgeBench}: A Benchmark for Evaluating {LLM}-Based Judges},
  author={Tan, Sijun and Zhuang, Siyuan and Montgomery, Kyle and Tang, William Y. and Cuadron, Alejandro and Wang, Chenguang and Popa, Raluca Ada and Stoica, Ion},
  booktitle={International Conference on Learning Representations (ICLR)},
  year={2025}
}

@inproceedings{bavaresco2025llms,
  title={{LLMs} Instead of Human Judges? A Large Scale Empirical Study Across 20 {NLP} Evaluation Tasks},
  author={Bavaresco, Anna and Bernardi, Raffaella and Bertolazzi, Leonardo and Elliott, Desmond and Fern{\'a}ndez, Raquel and Gatt, Albert and Ghaleb, Esam and Giulianelli, Mario and Hanna, Michael and Koller, Alexander and Martins, Andre and Mondorf, Philipp and Neplenbroek, Vera and Pezzelle, Sandro and Plank, Barbara and Schlangen, David and Suglia, Alessandro and Surikuchi, Aditya K. and Takmaz, Ece and Testoni, Alberto},
  booktitle={Proceedings of the 63rd Annual Meeting of the Association for Computational Linguistics (Volume 2: Short Papers)},
  pages={238--255},
  year={2025},
  doi={10.18653/v1/2025.acl-short.20},
  url={https://aclanthology.org/2025.acl-short.20/}
}

@book{billingsley2013convergence,
  title={Convergence of Probability Measures},
  author={Billingsley, Patrick},
  publisher={Wiley},
  year={1999},
  OPTdoi={10.1002/9780470316962}
}

@article{gu2024survey,
  title={A Survey on {LLM}-as-a-Judge},
  author={Gu, Jiawei and Jiang, Xuhui and Shi, Zhichao and Tan, Hexiang and Zhai, Xuehao and Xu, Chengjin and Li, Wei and Shen, Yinghan and Ma, Shengjie and Liu, Honghao and Wang, Saizhuo and Zhang, Kun and Wang, Yuanzhuo and Gao, Wen and Ni, Lionel and Guo, Jian},
  journal={The Innovation},
  volume={7},
  number={6},
  pages={101253},
  year={2026},
  doi={10.1016/j.xinn.2025.101253}
}

@misc{lykouris2024learning,
  title={Learning to Defer in Congested Systems: The {AI}-Human Interplay},
  author={Lykouris, Thodoris and Weng, Wentao},
  year={2024},
  eprint={2402.12237},
  archivePrefix={arXiv},
  primaryClass={cs.LG},
  url={https://arxiv.org/abs/2402.12237}
}

@misc{lee2024design,
  title={Design and Scheduling of an AI-based Queueing System},
  author={Lee, Jiung and Namkoong, Hongseok and Zeng, Yibo},
  year={2024},
  eprint={2406.06855},
  archivePrefix={arXiv},
  primaryClass={math.OC},
  url={https://arxiv.org/abs/2406.06855}
}

@misc{lin2026inference,
  title={Large-Scale {LLM} Inference with Heterogeneous Workloads: Prefill-Decode Contention and Asymptotically Optimal Control},
  author={Lin, Ruihan and Ding, Zezhen and Han, Zean and Zhang, Jiheng},
  year={2026},
  eprint={2602.02987},
  archivePrefix={arXiv},
  primaryClass={cs.DC},
  url={https://arxiv.org/abs/2602.02987}
}

@article{lee1997design,
  title={Design of a production system with a feedback buffer},
  author={Lee, Ho Woo and Seo, Dong Won},
  journal={Queueing Systems},
  volume={26},
  number={1--2},
  pages={187--202},
  year={1997},
  OPTdoi={10.1023/A:1019129107476}
}

@article{adve1994relationship,
  title={The Relationship Between {Bernoulli} and Fixed Feedback Policies for the {M/G/1} Queue},
  author={Adve, Vikram S. and Nelson, Randolph},
  journal={Operations Research},
  volume={42},
  number={2},
  pages={380--385},
  year={1994},
  OPTdoi={10.1287/opre.42.2.380}
}

@article{so1995optimal,
  title={Optimal Operating Policy for a Bottleneck with Random Rework},
  author={So, Kut C. and Tang, Christopher S.},
  journal={Management Science},
  volume={41},
  number={4},
  pages={620--636},
  year={1995},
  OPTdoi={10.1287/mnsc.41.4.620}
}

@article{chen1988empirical,
  title={Empirical Evaluation of a Queueing Network Model for Semiconductor Wafer Fabrication},
  author={Chen, Hong and Harrison, J. Michael and Mandelbaum, Avi and Van Ackere, Ann and Wein, Lawrence M.},
  journal={Operations Research},
  volume={36},
  number={2},
  pages={202--215},
  year={1988},
  OPTdoi={10.1287/opre.36.2.202}
}

@article{yom2014erlang,
  title={{Erlang-R}: A Time-Varying Queue with Reentrant Customers, in Support of Healthcare Staffing},
  author={Yom-Tov, Galit B. and Mandelbaum, Avishai},
  journal={Manufacturing \& Service Operations Management},
  volume={16},
  number={2},
  pages={283--299},
  year={2014},
  OPTdoi={10.1287/msom.2013.0474}
}

@article{chan2014use,
  title={When to Use Speedup: An Examination of Service Systems with Returns},
  author={Chan, Carri W. and Yom-Tov, Galit and Escobar, Gabriel},
  journal={Operations Research},
  volume={62},
  number={2},
  pages={462--482},
  year={2014},
  OPTdoi={10.1287/opre.2014.1258}
}

@inproceedings{jimenez2024swebench,
  title={{SWE}-bench: Can Language Models Resolve Real-World {GitHub} Issues?},
  author={Jimenez, Carlos E. and Yang, John and Wettig, Alexander and Yao, Shunyu and Pei, Kexin and Press, Ofir and Narasimhan, Karthik},
  booktitle={International Conference on Learning Representations},
  year={2024}
}

@misc{bjarnason2026randomness,
  title={On Randomness in Agentic Evals},
  author={Bjarnason, Bjarni Haukur and Silva, Andr{\'e} and Monperrus, Martin},
  year={2026},
  eprint={2602.07150},
  archivePrefix={arXiv},
  primaryClass={cs.SE},
  url={https://arxiv.org/abs/2602.07150},
  note={Data available at \url{https://doi.org/10.5281/zenodo.18684664}}
}

@misc{wang2026swereview,
  title={{SWE}-Review: Closing the Loop on Issue Resolution with Agentic Code Review},
  author={Wang, Ruoyu and Chen, Jierun and Wang, Shaowei and Tao, Chaofan and Yang, Sidi and Jiang, Yuxin and Yap, Kim-Hui and Shang, Lifeng and Li, Xiaohui and Bai, Haoli},
  year={2026},
  eprint={2607.06065},
  archivePrefix={arXiv},
  primaryClass={cs.SE},
  url={https://arxiv.org/abs/2607.06065}
}

@inproceedings{yang2024sweagent,
  title={{SWE}-agent: Agent-Computer Interfaces Enable Automated Software Engineering},
  author={Yang, John and Jimenez, Carlos E. and Wettig, Alexander and Lieret, Kilian and Yao, Shunyu and Narasimhan, Karthik and Press, Ofir},
  booktitle={Advances in Neural Information Processing Systems},
  volume={37},
  year={2024},
  url={https://proceedings.neurips.cc/paper_files/paper/2024/hash/5a7c947568c1b1328ccc5230172e1e7c-Abstract.html}
}

@inproceedings{zhang2024autocoderover,
  title={{AutoCodeRover}: Autonomous Program Improvement},
  author={Zhang, Yuntong and Ruan, Haifeng and Fan, Zhiyu and Roychoudhury, Abhik},
  booktitle={Proceedings of the 33rd ACM SIGSOFT International Symposium on Software Testing and Analysis},
  pages={1592--1604},
  year={2024},
  doi={10.1145/3650212.3680384},
  url={https://doi.org/10.1145/3650212.3680384}
}

@article{xia2025agentless,
  title={{Agentless}: Demystifying {LLM}-Based Software Engineering Agents},
  author={Xia, Chunqiu Steven and Deng, Yinlin and Dunn, Soren and Zhang, Lingming},
  journal={Proceedings of the ACM on Software Engineering},
  year={2025},
  doi={10.1145/3715754},
  url={https://doi.org/10.1145/3715754}
}

@inproceedings{bacchelli2013expectations,
  title={Expectations, Outcomes, and Challenges of Modern Code Review},
  author={Bacchelli, Alberto and Bird, Christian},
  booktitle={2013 35th International Conference on Software Engineering},
  pages={712--721},
  year={2013},
  doi={10.1109/ICSE.2013.6606617},
  url={https://doi.org/10.1109/ICSE.2013.6606617}
}

@inproceedings{tufano2021automating,
  title={Towards Automating Code Review Activities},
  author={Tufano, Rosalia and Pascarella, Luca and Tufano, Michele and Poshyvanyk, Denys and Bavota, Gabriele},
  booktitle={2021 IEEE/ACM 43rd International Conference on Software Engineering},
  pages={163--174},
  year={2021},
  doi={10.1109/ICSE43902.2021.00027},
  url={https://doi.org/10.1109/ICSE43902.2021.00027}
}

@inproceedings{li2022automating,
  title={Automating Code Review Activities by Large-Scale Pre-Training},
  author={Li, Zhiyu and Lu, Shuai and Guo, Daya and Duan, Nan and Jannu, Shailesh and Jenks, Grant and Majumder, Deep and Green, Jared and Svyatkovskiy, Alexey and Fu, Shengyu and Sundaresan, Neel},
  booktitle={Proceedings of the 30th ACM Joint European Software Engineering Conference and Symposium on the Foundations of Software Engineering},
  pages={1035--1047},
  year={2022},
  doi={10.1145/3540250.3549081},
  url={https://doi.org/10.1145/3540250.3549081}
}

@article{cui2026effects,
  title={The Effects of Generative {AI} on High-Skilled Work: Evidence from Three Field Experiments with Software Developers},
  author={Cui, Kevin Zheyuan and Demirer, Mert and Jaffe, Sonia and Musolff, Leon and Peng, Sida and Salz, Tobias},
  journal={Management Science},
  year={2026},
  doi={10.1287/mnsc.2025.00535},
  url={https://doi.org/10.1287/mnsc.2025.00535}
}

@inproceedings{liu2023geval,
  title={{G}-Eval: {NLG} Evaluation Using {GPT}-4 with Better Human Alignment},
  author={Liu, Yang and Iter, Dan and Xu, Yichong and Wang, Shuohang and Xu, Ruochen and Zhu, Chenguang},
  booktitle={Proceedings of the 2023 Conference on Empirical Methods in Natural Language Processing},
  pages={2511--2522},
  year={2023},
  doi={10.18653/v1/2023.emnlp-main.153},
  url={https://aclanthology.org/2023.emnlp-main.153/}
}

@inproceedings{kim2024prometheus,
  title={Prometheus 2: An Open Source Language Model Specialized in Evaluating Other Language Models},
  author={Kim, Seungone and Suk, Juyoung and Longpre, Shayne and Lin, Bill Yuchen and Shin, Jamin and Welleck, Sean and Neubig, Graham and Lee, Moontae and Lee, Kyungjae and Seo, Minjoon},
  booktitle={Proceedings of the 2024 Conference on Empirical Methods in Natural Language Processing},
  pages={4334--4353},
  year={2024},
  doi={10.18653/v1/2024.emnlp-main.248},
  url={https://aclanthology.org/2024.emnlp-main.248/}
}

@inproceedings{wang2024notfair,
  title={Large Language Models Are Not Fair Evaluators},
  author={Wang, Peiyi and Li, Lei and Chen, Liang and Cai, Zefan and Zhu, Dawei and Lin, Binghuai and Cao, Yunbo and Kong, Lingpeng and Liu, Qi and Liu, Tianyu and Sui, Zhifang},
  booktitle={Proceedings of the 62nd Annual Meeting of the Association for Computational Linguistics},
  pages={9440--9450},
  year={2024},
  doi={10.18653/v1/2024.acl-long.511},
  url={https://aclanthology.org/2024.acl-long.511/}
}

@inproceedings{koo2024cognitive,
  title={Benchmarking Cognitive Biases in Large Language Models as Evaluators},
  author={Koo, Ryan and Lee, Minhwa and Raheja, Vipul and Park, Jong Inn and Kim, Zae Myung and Kang, Dongyeop},
  booktitle={Findings of the Association for Computational Linguistics: ACL 2024},
  pages={517--545},
  year={2024},
  doi={10.18653/v1/2024.findings-acl.29},
  url={https://aclanthology.org/2024.findings-acl.29/}
}

@inproceedings{madras2018predict,
  title={Predict Responsibly: Improving Fairness and Accuracy by Learning to Defer},
  author={Madras, David and Pitassi, Toniann and Zemel, Richard},
  booktitle={Advances in Neural Information Processing Systems},
  volume={31},
  pages={6147--6157},
  year={2018}
}

@inproceedings{mozannar2020consistent,
  title={Consistent Estimators for Learning to Defer to an Expert},
  author={Mozannar, Hussein and Sontag, David},
  booktitle={Proceedings of the 37th International Conference on Machine Learning},
  series={Proceedings of Machine Learning Research},
  volume={119},
  pages={7076--7087},
  year={2020},
  publisher={PMLR},
  url={https://proceedings.mlr.press/v119/mozannar20b.html}
}

@inproceedings{raghu2019uncertainty,
  title={Direct Uncertainty Prediction for Medical Second Opinions},
  author={Raghu, Maithra and Blumer, Katy and Sayres, Rory and Obermeyer, Ziad and Kleinberg, Bobby and Mullainathan, Sendhil and Kleinberg, Jon},
  booktitle={Proceedings of the 36th International Conference on Machine Learning},
  series={Proceedings of Machine Learning Research},
  volume={97},
  pages={5281--5290},
  year={2019},
  publisher={PMLR},
  url={https://proceedings.mlr.press/v97/raghu19a.html}
}

@inproceedings{donahue2022collaboration,
  title={Human-Algorithm Collaboration: Achieving Complementarity and Avoiding Unfairness},
  author={Donahue, Kate and Chouldechova, Alexandra and Kenthapadi, Krishnaram},
  booktitle={Proceedings of the 2022 ACM Conference on Fairness, Accountability, and Transparency},
  pages={1639--1656},
  year={2022},
  doi={10.1145/3531146.3533221},
  url={https://doi.org/10.1145/3531146.3533221}
}

@article{kumar1993reentrant,
  title={Re-Entrant Lines},
  author={Kumar, P. R.},
  journal={Queueing Systems},
  volume={13},
  number={1--3},
  pages={87--110},
  year={1993},
  doi={10.1007/BF01158930}
}

@article{narahari1996modeling,
  title={Modeling Reentrant Manufacturing Systems with Inspection Stations},
  author={Narahari, Y. and Khan, L. M.},
  journal={Journal of Manufacturing Systems},
  volume={15},
  number={6},
  pages={367--378},
  year={1996},
  doi={10.1016/S0278-6125(97)83051-8}
}

@article{chen1993dynamic,
  title={Dynamic Scheduling of a Multiclass Fluid Network},
  author={Chen, Hong and Yao, David D.},
  journal={Operations Research},
  volume={41},
  number={6},
  pages={1104--1115},
  year={1993},
  doi={10.1287/opre.41.6.1104}
}
\begin{APPENDICES}
\normalsize
\begingroup
\def\MainAppendix{}
\ifdefined\MainAppendix
\section{Model Details and Proof Sketches}
\label{app:analytical_proofs}

This appendix states the stochastic state equations and the fluid model,
and sketches the proofs of all results in
Sections~\ref{sec:fluid_model}--\ref{sec:stoch_control}.  Complete proofs
are in the Online Supplement (planning results in Section~EC.1,
Theorem~\ref{thm:fluid-limit} in Section~EC.2,
Theorem~\ref{thm:asymptotic_optimality} in Section~EC.3).
\else
\ECSwitch
\ECHead{Online Supplement}
\section{Complete Proofs of the Planning Results}
\label{ec:planning_proofs}

This section proves Lemma~\ref{prop:approved_flow_reduction} and
Propositions~\ref{prop:four_phase}, \ref{thm:multiclass_path},
and~\ref{prop:reviewed_switch}, sketched in the Appendix of the main paper;
equation numbers such as (\ref{eq:activity_lp}) refer to the main paper.
Throughout, we write \eqref{eq:activity_lp} in the useful-service
variables \(x_i,v_i\) of Section~\ref{subsec:single_class}:
\begin{subequations}\label{eq:mc_reduced}
\begin{align}
\max\quad&\sum_i r_i\mu_{g,i}(1-\alpha_i)
(x_i-\beta_i^{(I)}v_i)\label{eq:mc_reduced_obj}\\
\text{s.t.}\quad&0\le v_i\le x_i,\quad \sum_ix_i\le n_g,\label{eq:mc_reduced_box}\\
&\sum_i\frac{\mu_{g,i}}{\mu_{r,i}}v_i\le n_r,\qquad
\sum_i\frac{\mu_{g,i}}{\mu_{h,i}}
(x_i-p_{\mathrm{rej},i}v_i)\le n_h,\label{eq:mc_reduced_capacity}\\
&x_i-\beta_i^{(I)}v_i
\le\frac{\lambda_i}{\mu_{g,i}(1-\alpha_i)},\quad i\in\mathcal I.
\label{eq:mc_reduced_qw_nonneg}
\end{align}
\end{subequations}
Substituting \(f_i^D=\mu_{g,i}(1-\alpha_i)(x_i-v_i)\) and
\(f_i^S=\mu_{g,i}(1-\alpha_i)(1-\beta_i^{(I)})v_i\) into
\eqref{eq:activity_lp} gives \eqref{eq:mc_reduced}.
\fi

\ifdefined\MainAppendix
\subsection{Stochastic State Equations}
\label{app:state_equations}

Let \(\mathcal N=\{g,r,h_d,h_s\}\) denote the generation, AI-review, direct
human-approval, and post-review human-approval nodes.  For node \(\ell\) and
class \(i\), \(Q_{\ell,i}^n\), \(X_{\ell,i}^n\), and \(U_{\ell,i}^n\) are the
waiting count, in-service count, and cumulative service admissions;
\(S_{\ell\to m,i}^n\) counts completions at \(\ell\) routed to successor
\(m\); and \(B_{\ell,i}^n\) counts obsolescence.  With
\(A_{g,i}^n=A_i^n\), \(A_{\ell,i}^n=0\) for \(\ell\ne g\), and
predecessor and successor sets \(\mathcal I_\ell,\mathcal O_\ell\), the
balances are
\begin{align}
Q_{\ell,i}^n(t)
&=Q_{\ell,i}^n(0)+A_{\ell,i}^n(t)
+\sum_{k\in\mathcal I_\ell}S_{k\to\ell,i}^n(t)
-U_{\ell,i}^n(t)-B_{\ell,i}^n(t),\label{eq:pf_balance_compact}\\
X_{\ell,i}^n(t)
&=X_{\ell,i}^n(0)+U_{\ell,i}^n(t)
-\sum_{m\in\mathcal O_\ell}S_{\ell\to m,i}^n(t),\qquad \ell\in\mathcal N,
\nonumber
\end{align}
and the capacity constraints are
\begin{equation}
\sum_iX_{g,i}^n(t)\le\lfloor nn_g\rfloor,\quad
\sum_iX_{r,i}^n(t)\le\lfloor nn_r\rfloor,\quad
\sum_i\bigl(X_{h_d,i}^n(t)+X_{h_s,i}^n(t)\bigr)\le\lfloor nn_h\rfloor.
\label{eq:pf_capacity_constraints}
\end{equation}
The uncontrolled outcome probabilities are
\((p_{r\to g,i},p_{r\to h_s,i})=(p_{\mathrm{rej},i},p_{\mathrm{pass},i})\),
\((p_{h_d\to g,i},p_{h_d\to c,i})=(\alpha_i,1-\alpha_i)\), and
\((p_{h_s\to g,i},p_{h_s\to c,i})=(1-q_{\mathrm{acc},i},q_{\mathrm{acc},i})\),
where \(c\) denotes approval.  Mutually independent unit-rate Poisson
processes give the random-time-change representation
\(A_i^n(t)=N_{A,i}(n\lambda_it)\),
\(B_{\ell,i}^n(t)=N_{B,\ell,i}(\int_0^t\theta_{\ell,i}Q_{\ell,i}^n(s)\,ds)\),
and
\(S_{\ell\to m,i}^n(t)=N_{\ell\to m,i}(\int_0^t\mu_{\ell,i}p_{\ell\to m,i}
X_{\ell,i}^n(s)\,ds)\) for \(\ell\ne g\).
Generation completions, with compensator
\(\int_0^t\mu_{g,i}X_{g,i}^n(s)\,ds\), are split pathwise by the routing
control into \(S_{g\to r,i}^n+S_{g\to h_d,i}^n\); approvals are
\(C_i^n=S_{h_d\to c,i}^n+S_{h_s\to c,i}^n\), and the main text writes
\(X_i^n,Y_i^n,Z_{i,d}^n,Z_{i,s}^n\) for \(X_{g,i}^n,X_{r,i}^n,X_{h_d,i}^n,
X_{h_s,i}^n\).

\subsection{Fluid Model}
\label{app:fluid_model}

Scale every count and cumulative flow by \(n\) and write limits in
lowercase.  Any subsequential fluid limit obeys the node balances
\begin{align}
q_{\ell,i}(t)
&=q_{\ell,i}(0)+a_{\ell,i}(t)
+\sum_{k\in\mathcal I_\ell}s_{k\to\ell,i}(t)
-u_{\ell,i}(t)-b_{\ell,i}(t),\label{eq:fluid-balance}\\
x_{\ell,i}(t)
&=x_{\ell,i}(0)+u_{\ell,i}(t)
-\sum_{m\in\mathcal O_\ell}s_{\ell\to m,i}(t),\nonumber
\end{align}
with \(a_{g,i}(t)=\lambda_it\),
\(b_{\ell,i}(t)=\int_0^t\theta_{\ell,i}q_{\ell,i}(s)\,ds\), and
\begin{equation}
\dot s_{\ell\to m,i}(t)
=\mu_{\ell,i}p_{\ell\to m,i}x_{\ell,i}(t)\quad(\ell\ne g),\qquad
\dot s_{g\to r,i}(t)=\mu_{g,i}v_i(t),\quad
\dot s_{g\to h_d,i}(t)=\mu_{g,i}(x_i(t)-v_i(t)),
\label{eq:fluid-rates}
\end{equation}
where \(0\le v_i(t)\le x_i(t):=x_{g,i}(t)\) is the generation in service
routed to review.  The three resource pools obey the capacity constraints
\begin{equation}
\sum_ix_i(t)\le n_g,\quad \sum_ix_{r,i}(t)\le n_r,\quad
\sum_i(x_{h_d,i}(t)+x_{h_s,i}(t))\le n_h.
\label{eq:fluid-cap}
\end{equation}
The associated stationary planning program is
\eqref{eq:full_fluid_program} in Section~\ref{sec:fluid_model}.

\subsection{Proof Sketch of Theorem~\ref{thm:fluid-limit}}
\label{app:proof_fluid_limit}

Capacity bounds make every service clock \(O(nT)\).  Queue balances then
keep scaled waiting mass \(O_P(1)\) on \([0,T]\), so every coordinate of
the fluid-scaled process is stochastically bounded.  Admissible policies
move finitely many jobs per event, hence every scaled jump is \(O(1/n)\);
a localized modulus argument gives \(C\)-tightness.  The FSLLN and the
random-time-change theorem identify uncontrolled flows; the controlled
split is a measurable fraction of generation completions.  Passing to the
limit in the node balances yields
\eqref{eq:fluid-balance}--\eqref{eq:fluid-cap}.

\subsection{Proof Sketch of Lemma~\ref{prop:approved_flow_reduction}}
\label{app:proof_reduction}

Given a feasible point, the approved flows
\(f_i^D=\mu_{h,i}(1-\alpha_i)\hat z_{i,d}\) and
\(f_i^S=\mu_{h,i}q_{\mathrm{acc},i}\hat z_{i,s}\) satisfy
\eqref{eq:activity_lp}: obsolescence only removes work, so service at each
node is at most the output it receives, the mass balance of class \(i\)
bounds its approval rate by \(\lambda_i\), and the objectives agree.
Conversely, given feasible \((f^D,f^S)\), run the service levels that carry
these flows, leave downstream queues empty, and hold the excess arrivals in
the generation queue at \(\hat q_{g,i}=(\lambda_i-f_i^D-f_i^S)/\theta_i\).

\subsection{Proof Sketch of Proposition~\ref{prop:four_phase}}
\label{app:proof_four_phase}

Under Assumption~\ref{ass:overloaded_single} the arrival cap is redundant
and the problem is \eqref{eq:sc_lp}.  For fixed reviewed volume \(v\) the
objective increases in \(x\), so \(x(v)=\min\{n_g,E+p_{\mathrm{rej}}v\}\);
by Assumption~\ref{ass:informative_single} the objective
\(E+(p_{\mathrm{rej}}-\beta^{(I)})v\) then rises in \(v\) while generation
is slack and equals \(n_g-\beta^{(I)}v\), falling in \(v\), once it binds.
Hence \(v^*=\min\{J_{\mathrm{eff}},\,E/p_{\mathrm{pass}},\,(n_g-E)/p_{\mathrm{rej}}\}\),
the thresholds in \eqref{eq:sc_thresholds_123} are the \(n_h\) at which the
binding cap changes, and \(\phi^*=v^*/x^*\) gives the four expressions.

\subsection{Proof Sketches of Propositions~\ref{thm:multiclass_path}
and~\ref{prop:reviewed_switch}}
\label{app:proof_multiclass}

\emph{Proposition~\ref{thm:multiclass_path}.}  Part~1 is LP duality: the dual feasible set does not depend on \(s\), so
\(V(s)\) is the minimum of finitely many affine functions of \(s\), hence
concave and piecewise linear, and it is nondecreasing because the primal
feasible set grows with \(s\); the maintainer price \(\chi\) is a
supergradient of \(V\).  Part~2 uses that a
fixed optimal basis expresses the basic variables as an affine function of
the right-hand side, which is affine in \(s\).  For the sparsity bound,
count positive route variables at an extreme point, one per active class
plus one per class using both routes; there are at most as many as binding
constraints other than nonnegativity, three resource constraints plus one
cap per capped class, and removing the capped classes from both sides
leaves \(U(s)+L(s)\le3\).  For Part~3, as \(h\downarrow0\) only the
maintainer constraint binds, a fractional knapsack in reward per maintainer
hour; as \(h\to\infty\) it drops and the direct route earns the same reward
with weakly less generation and no AI review.

\emph{Proposition~\ref{prop:reviewed_switch}.}\label{app:proof_reviewed_switch}
With common parameters, \eqref{eq:sc_lp} extends class by class and two
upper bounds do the work: reward per unit of maintainer capacity is at most
\(q_{\mathrm{acc},i_A}/(1-\alpha)\), so reward is at most
\(q_{\mathrm{acc},i_A}E/(1-\alpha)\), attained by full review of class
\(i_A\) when it fits; and relieving the maintainer constraint by one unit
costs at least \(\eta_{i_B}<1\), so reward is at most
\(n_g-\eta_{i_B}(n_g-E)\), attained by filling generation and reviewing only
class \(i_B\) when it fits.  The bounds are affine in \(E\) with different
slopes, which orders the intervals.

\subsection{Proof Sketch of Theorem~\ref{thm:asymptotic_optimality}}
\label{app:sketch_asymptotic_opt}

\emph{Step 1 (upper bound).}  For any admissible policy,
\(\mathbb E[C_i^n(T)]/(nT)=\mu_{h,i}[(1-\alpha_i)\hat z_{i,d}+
q_{\mathrm{acc},i}\hat z_{i,s}]\), with \(\hat z\) the expected
time-average maintainer-service levels; summing the node balances shows
that the effort feeding this service is feasible for \eqref{eq:activity_lp}
up to \(O(1/(nT))\), so no policy exceeds \(R^*\) in the iterated limit.
Rate stability is not assumed.

\emph{Step 2 (stabilization converges).}  The cumulative quota keeps
\(|H_i^n-\phi_i^*G_i^n|<1\), so every fluid limit routes the planned
fraction, and a reservation below its target either has an empty queue or
admits at full rate.  For each class, the return-weighted deficit
\(L_i(t)=(x_i^*-x_i)^+ +\rho_i(y_i^*-y_i)^+
+\alpha_i(z_{i,d}^*-z_{i,d})^+ +(1-q_{\mathrm{acc},i})(z_{i,s}^*-z_{i,s})^+\),
with \(\rho_i\) the probability that a reviewed task returns to
generation, is a Lyapunov function: each downstream deficit is fed only by
the one upstream, and below its target the generation coordinate admits all
input, so the target balance leaves a drift of \(-c_id_x\) with
\(c_i=\mu_{g,i}(1-\alpha_i)(1-\beta_i^{(I)}\phi_i^*)>0\).  No overload
condition enters.

\emph{Step 3 (pooled mode holds the target).}  Since \(\delta_n\to0\), a fluid
trajectory leaves stabilization only at the target, where a positive effort
deficit leaves a backlog of the same fluid size, so the coordinate is
nonempty and the largest-deficit rule keeps each cumulative service effort
within \(o(n)\) of \(nb_c^*t\), the idle coordinate absorbing planned
idleness; every fluid limit thus has time-averaged service at the target,
and a deviation above \(2n\delta_n\) returns the policy to Step~2.
\emph{Step 4 (lower bound).}  Along a subsequence of \(n\) attaining the
\(\liminf\) of the scaled reward on \([0,T]\), Steps~2 and~3 put the fluid
limit's reward at target after a transient independent of \(n\); letting
\(T\to\infty\) with Step~1 gives \eqref{eq:asymptotic_opt}.
\fi
\ifdefined\MainAppendix\else
\subsection{Proof of Lemma~\ref{prop:approved_flow_reduction}}
\label{ec:proof_reduction}

\begin{proof}{Proof of Lemma~\ref{prop:approved_flow_reduction}.}
Let a feasible point of the full program \eqref{eq:full_fluid_program} be given.  Define the approved
flows \(f_i^D=\mu_{h,i}(1-\alpha_i)\hat z_{i,d}\),
\(f_i^S=\mu_{h,i}q_{\mathrm{acc},i}\hat z_{i,s}\) and the useful-service
variables
\[
v_i:=\frac{\mu_{h,i}}{\mu_{g,i}p_{\mathrm{pass},i}}\,\hat z_{i,s},
\qquad
x_i:=v_i+\frac{\mu_{h,i}}{\mu_{g,i}}\,\hat z_{i,d}.
\]
These are the generation and review efforts that feed maintainer service.
Downstream obsolescence only weakens the balances: maintainer service on the
direct route cannot exceed direct generation output, maintainer service on
the screened route cannot exceed AI-approved output, and AI-review service
cannot exceed the generation output routed to review.  Hence
\(0\le v_i\le x_i\), \(\sum_ix_i\le n_g\), and
\(\sum_i(\mu_{g,i}/\mu_{r,i})v_i\le n_r\), while
\[
\sum_i\frac{\mu_{g,i}}{\mu_{h,i}}
\bigl(x_i-p_{\mathrm{rej},i}v_i\bigr)
=\sum_i(\hat z_{i,d}+\hat z_{i,s})\le n_h.
\]
The class-\(i\) approval rate of the full point is
\(\mu_{h,i}[(1-\alpha_i)\hat z_{i,d}+q_{\mathrm{acc},i}\hat z_{i,s}]
=\mu_{g,i}(1-\alpha_i)(x_i-\beta_i^{(I)}v_i)\), using
\(p_{\mathrm{pass},i}q_{\mathrm{acc},i}=(1-\alpha_i)(1-\beta_i^{(I)})\).
The system-wide class-\(i\) mass balance bounds this rate by \(\lambda_i\),
because every other external departure is an obsolescence.  Thus
\((x_i,v_i)_i\) is feasible for \eqref{eq:mc_reduced}, its image
\((f^D,f^S)\) is feasible for \eqref{eq:activity_lp}, and the objectives
agree.  Therefore \(V_{\mathrm{full}}(n_h)\le V(n_h)\).

Conversely, take any feasible \((f^D,f^S)\) in \eqref{eq:activity_lp}, set
\[
v_i=\frac{f_i^S}{\mu_{g,i}(1-\alpha_i)(1-\beta_i^{(I)})},
\qquad
x_i=v_i+\frac{f_i^D}{\mu_{g,i}(1-\alpha_i)},
\]
and reconstruct
\[
\hat y_i=\frac{\mu_{g,i}}{\mu_{r,i}}v_i,\qquad
\hat z_{i,d}=\frac{\mu_{g,i}}{\mu_{h,i}}(x_i-v_i),\qquad
\hat z_{i,s}=\frac{\mu_{g,i}}{\mu_{h,i}}p_{\mathrm{pass},i}v_i,
\]
with \(\hat x_i=x_i\), \(\hat v_i=v_i\), empty downstream queues, and
\[
\hat q_{g,i}:=
\frac{\lambda_i-\mu_{g,i}(1-\alpha_i)
(x_i-\beta_i^{(I)}v_i)}{\theta_i}\ge0.
\]
The stationary balances hold, downstream obsolescence is zero, and the
full-program objective equals \(\sum_ir_i(f_i^D+f_i^S)\).  Therefore
\(V_{\mathrm{full}}(n_h)\ge V(n_h)\), proving equality.
\end{proof}

\subsection{Proof of Proposition~\ref{prop:four_phase}}
\label{ec:proof_four_phase}

\begin{proof}{Proof.}
The thresholds in \eqref{eq:sc_thresholds_123} satisfy
\[
n_h^{(2)}-n_h^{(1)}
=\frac{\mu_g}{\mu_h}(n_g-J_{\mathrm{eff}})\ge0,\qquad
n_h^{(3)}-n_h^{(2)}
=\frac{\mu_g}{\mu_h}p_{\mathrm{rej}}J_{\mathrm{eff}}\ge0,
\]
because \(p_{\mathrm{pass}}+p_{\mathrm{rej}}=1\); both inequalities are
strict when \(0<J_{\mathrm{eff}}<n_g\), since \(p_{\mathrm{rej}}>0\).

Under Assumption~\ref{ass:overloaded_single} the arrival constraint is
redundant, so the problem is \eqref{eq:sc_lp}.  For fixed \(v\) the
objective increases in \(x\), so \(x(v)=\min\{n_g,E+p_{\mathrm{rej}}v\}\).
Feasibility of \(v\le x(v)\) together with the review constraint is
equivalent to \(0\le v\le\min\{J_{\mathrm{eff}},E/p_{\mathrm{pass}}\}\).
Assumption~\ref{ass:informative_single} gives
\[
p_{\mathrm{rej}}-\beta^{(I)}
=\frac{p_{\mathrm{pass}}}{1-\alpha}
\bigl(q_{\mathrm{acc}}-(1-\alpha)\bigr)>0,
\]
so \(x(v)-\beta^{(I)}v=E+(p_{\mathrm{rej}}-\beta^{(I)})v\) increases in
\(v\) while the generation constraint is slack and equals
\(n_g-\beta^{(I)}v\), decreasing in \(v\), afterward.  Hence, for
\(E<n_g\),
\[
v^*=\min\!\left\{J_{\mathrm{eff}},\frac{E}{p_{\mathrm{pass}}},
\frac{n_g-E}{p_{\mathrm{rej}}}\right\},
\qquad x^*=\min\{n_g,E+p_{\mathrm{rej}}v^*\},
\]
and \(E\ge n_g\) gives \((x^*,v^*)=(n_g,0)\).  The threshold ordering
splits this formula into the four intervals of the proposition, on which
\[
(x^*,v^*)=
\left(\frac{E}{p_{\mathrm{pass}}},\frac{E}{p_{\mathrm{pass}}}\right),\quad
(E+p_{\mathrm{rej}}J_{\mathrm{eff}},J_{\mathrm{eff}}),\quad
\left(n_g,\frac{n_g-E}{p_{\mathrm{rej}}}\right),\quad
(n_g,0),
\]
respectively.  Taking \(\phi^*=v^*/x^*\) gives the four expressions.
\end{proof}

\subsection{Proof of Proposition~\ref{thm:multiclass_path}}
\label{ec:proof_multiclass}

\begin{proof}{Proof.}
Index routes by \(k=(i,m)\).  Write \(A\) for the three-row matrix whose
columns are the resource vectors \(a_i^m\), \(B\) for the class-incidence
matrix, and \(c\) for the vector of route rewards \(r_i\).  By LP duality,
\[
V(s)=\min_{\pi,\nu\ge0}
\left\{n_g\omega+n_r\kappa+s\chi+\lambda^\top\nu:
A^\top\pi+B^\top\nu\ge c\right\},
\quad \pi=(\omega,\kappa,\chi).
\]
The dual feasible set does not depend on \(s\) and has finitely many
extreme points, so \(V\) is the pointwise minimum of finitely many affine
functions of \(s\): it is concave and piecewise linear, and it is
nondecreasing because the primal feasible set expands with \(s\).  The
maintainer price \(\chi\) is a supergradient of \(V\), so one-sided optimal
prices can be selected nonincreasing.  The kinks of \(V\), together with
the finitely many points at which an optimal basis changes without changing
the affine piece, form the partition \(s_0<\cdots<s_{K+1}\); the slope on
each piece is the selected price \(\chi_k\).  This proves Part~1.

On an interval where one primal--dual basis pair remains optimal, the basic
variables equal the inverse basis matrix times a right-hand side that is
affine in \(s\); hence the allocation is affine in \(s\), and the active
routes and prices are constant.  This gives the form \(u_k+s\,d_k\).

For sparsity, select an optimal extreme point.  Let \(C\) be the number of
active classes at their arrival caps, \(U\) the number of active classes
strictly below their caps, and \(L\) the number of active classes using
both routes.  The number of positive route variables is \(P=C+U+L\).  At an
extreme point, \(P\) is at most the number of linearly independent binding
constraints other than nonnegativity, which is at most the three resource
constraints plus the \(C\) binding arrival caps.  Thus \(P\le C+3\), which
is \eqref{eq:multiclass_sparsity}.

For Part~3, note that every route consumes positive maintainer capacity.
As \(h\downarrow0\), all flows are \(O(h)\), so the generation, AI-review,
and arrival constraints are slack for sufficiently small \(h\).  The
remaining problem is a fractional knapsack with maintainer capacity as its
only binding resource, and it uses only routes maximizing
\(r_k/a_{e,k}\).  At the other endpoint, remove the maintainer constraint.
Within each class, the direct route earns the same reward as the screened
route, uses weakly less generation because \(1-\beta_i^{(I)}\le1\), and
uses no AI-review capacity; replacing screened flow by direct flow
therefore yields a direct-only optimum.  Its maintainer use is finite, and
for all \(h\) above that level it remains feasible and optimal.
\end{proof}

\subsection{Proof of Proposition~\ref{prop:reviewed_switch}}
\label{ec:proof_reviewed_switch}

\begin{proof}{Proof of Proposition~\ref{prop:reviewed_switch}.}
With common \(r,\mu_g,\mu_r,\mu_h,\alpha\), dividing
\eqref{eq:mc_reduced} by \(r\mu_g(1-\alpha)\) and using the units
\eqref{eq:sc_normalized_capacities} gives
\(\max\sum_i(x_i-\beta_i^{(I)}v_i)\) subject to \(\sum_ix_i\le n_g\),
\(\sum_iv_i\le J\), \(\sum_i(x_i-p_{\mathrm{rej},i}v_i)\le E\), and
\(0\le v_i\le x_i\); Assumption~\ref{ass:overloaded_single} makes the
arrival caps slack.  Since
\(p_{\mathrm{rej},i}=(1-\alpha)\beta_i^{(I)}+\alpha(1-\beta_i^{(II)})\),
\(\eta_i=\beta_i^{(I)}/p_{\mathrm{rej},i}<1\) if and only if
\(1-\beta_i^{(II)}>\beta_i^{(I)}\), which is
Assumption~\ref{ass:informative_single}; \(\eta_i\) is decreasing in
\((1-\beta_i^{(II)})/\beta_i^{(I)}\), defects caught per merge-ready change
sent back.  The number of classes plays no role below.

\emph{Bound A.}  Split \(x_i=(x_i-v_i)+v_i\).  The direct part earns
\(x_i-v_i\) and uses maintainer capacity \(x_i-v_i\); the screened part
earns \((1-\beta_i^{(I)})v_i\) and uses \(p_{\mathrm{pass},i}v_i\).  Since
\((1-\beta_i^{(I)})/p_{\mathrm{pass},i}=q_{\mathrm{acc},i}/(1-\alpha)>1\)
under Assumption~\ref{ass:informative_single}, every feasible point earns
at most \(q_{\mathrm{acc},i_A}E/(1-\alpha)\), with equality only if all
maintainer capacity is used by screened class-\(i_A\) work.  When
\(E\le p_{\mathrm{pass},i_A}J_{\mathrm{eff}}\), the point
\(x_{i_A}=v_{i_A}=E/p_{\mathrm{pass},i_A}\), all else zero, is feasible
(\(v_{i_A}\le J_{\mathrm{eff}}=\min\{n_g,J\}\)) and attains the bound.
Strictness of the \(q_{\mathrm{acc}}\) ranking makes it the unique optimum.
This is Part~1.

\emph{Bound B.}  Let \(X=\sum_ix_i\le n_g\).  The maintainer constraint
gives \(\sum_ip_{\mathrm{rej},i}v_i\ge X-E\), so
\(\sum_i\beta_i^{(I)}v_i\ge\eta_{i_B}(X-E)^+\), and the reward is at most
\(X-\eta_{i_B}(X-E)^+\).  Because \(\eta_{i_B}<1\), this is increasing in
\(X\), hence at most \(n_g-\eta_{i_B}(n_g-E)\), with equality only if
\(X=n_g\), \(\sum_ip_{\mathrm{rej},i}v_i=n_g-E\), and \(v_i=0\) for
\(i\ne i_B\) by strictness of the ranking.  When
\(n_g-p_{\mathrm{rej},i_B}J_{\mathrm{eff}}\le E<n_g\), the point
\(v_{i_B}=(n_g-E)/p_{\mathrm{rej},i_B}\le J_{\mathrm{eff}}\),
\(x_{i_B}\ge v_{i_B}\), \(\sum_ix_i=n_g\), \(v_i=0\) otherwise, is feasible
and attains the bound.  This is Part~2; the split of direct effort between
the classes is immaterial because rewards are equal.

\emph{Ordering.}  Both bounds hold for every \(E\).  If both candidate
points were feasible on an interval of \(E\), both would be optimal there
and the two bounds would coincide on that interval.  They are affine in
\(E\) with slopes \(q_{\mathrm{acc},i_A}/(1-\alpha)>1\) and
\(\eta_{i_B}<1\), so they agree at no more than one point.  Hence
\(p_{\mathrm{pass},i_A}J_{\mathrm{eff}}\le n_g-p_{\mathrm{rej},i_B}J_{\mathrm{eff}}\).
In the language of Section~\ref{subsec:multiclass}, Part~1 is the regime
with only \(\chi>0\), where the index \eqref{eq:general_priority_index}
ranks by \(r_i/a_{h,i}^S\); Part~2 has \(\omega,\chi>0\) and \(\kappa=0\),
where it ranks screened routes by generation used per unit of relief.
\end{proof}

\subsection{Two Extensions of the Error Model}
\label{ec:revision_benefit}

\paragraph{Stage-dependent revision quality.}
Suppose that an output revised after an AI rejection is merge-ready with
probability \(1-\alpha_i'\) rather than \(1-\alpha_i\), while an output
regenerated after a maintainer rejection keeps probability \(1-\alpha_i\).  On the screened
route each attempt is then fresh (\(F\)) or guided (\(R\)).  Write
\(s(\alpha):=(1-\alpha)(1-\beta_i^{(I)})\),
\(\rho(\alpha):=(1-\alpha)\beta_i^{(I)}+\alpha(1-\beta_i^{(II)})\), and
\(\kappa(\alpha):=\alpha\beta_i^{(II)}\) for the probabilities that an
attempt is approved, returned by the reviewer (to stage \(R\)), or returned
by the maintainer (to stage \(F\)); a fresh attempt uses \(\alpha_i\) and a
guided attempt \(\alpha_i'\).  The expected numbers of attempts
\(m_F,m_R\) and of maintainer inspections \(e_F,e_R\) until approval solve
\begin{align*}
m_F&=1+\rho(\alpha_i)\,m_R+\kappa(\alpha_i)\,m_F, &
m_R&=1+\rho(\alpha_i')\,m_R+\kappa(\alpha_i')\,m_F,\\
e_F&=s(\alpha_i)+\kappa(\alpha_i)+\rho(\alpha_i)\,e_R+\kappa(\alpha_i)\,e_F, &
e_R&=s(\alpha_i')+\kappa(\alpha_i')+\rho(\alpha_i')\,e_R+\kappa(\alpha_i')\,e_F,
\end{align*}
and the screened activity vector becomes
\(a_i^S=(m_F/\mu_{g,i},\,m_F/\mu_{r,i},\,e_F/\mu_{h,i})\), while
\(a_i^D\) is unchanged; with \(\alpha_i'=\alpha_i\) this reduces to
\(m_F=1/(p_{\mathrm{pass},i}q_{\mathrm{acc},i})\) and
\(e_F=1/q_{\mathrm{acc},i}\), the coefficients of \eqref{eq:activity_lp}.
Lemma~\ref{prop:approved_flow_reduction} and Parts~1 and~2 of
Proposition~\ref{thm:multiclass_path} use activity vectors only through
finiteness and positivity and hold verbatim.  The throttling phase of
Proposition~\ref{prop:four_phase} and the direct-only endpoint in Part~3 of
Proposition~\ref{thm:multiclass_path} use the sign of
\(a_{g,i}^S-a_{g,i}^D\), positive under independent attempts whenever
\(\beta_i^{(I)}>0\).  If \(a_{g,i}^S\le a_{g,i}^D\), the
screened route weakly dominates the direct route in every coordinate except
review capacity, so review of class \(i\) is limited only by review
capacity and neither statement applies to that class.  For the calibration
of Section~\ref{subsec:calibration}, \(a_{g,i}^S\le a_{g,i}^D\) requires
\(1-\alpha_i'\) to exceed 0.416 for the low-quality tier and 0.864 for the
medium-quality tier, and it cannot hold for the high-quality tier; after one
guided revision the source reports 56.9\% for the low-quality tier.

\paragraph{Imperfect maintainers.}
Suppose a maintainer approves a defective change with probability
\(\gamma_i\in[0,1)\) and each escaped defect costs \(c_i\ge0\).  A direct
attempt is approved with probability \(m_i+\alpha_i\gamma_i\), earning
\(r_i\) with probability \(m_i\) and \(-c_i\) with probability
\(\alpha_i\gamma_i\), and is returned otherwise.  A screened attempt reaches
the maintainer with probability \(p_{\mathrm{pass},i}\), is approved
merge-ready with probability \(m_it_i\), approved defective with
probability \(\alpha_i\beta_i^{(II)}\gamma_i\), and returned otherwise.
Per unit of generation in use, expected reward is therefore
\(\mu_{g,i}(r_im_i-c_i\alpha_i\gamma_i)\) on the direct route and
\(\mu_{g,i}(r_im_it_i-c_i\alpha_i\beta_i^{(II)}\gamma_i)\) on the screened
route, while review and maintainer use per unit of generation are exactly
those of \eqref{eq:mc_reduced}.  With
\(\tilde r_i:=r_im_i-c_i\alpha_i\gamma_i>0\) and
\[
\tilde\beta_i:=\frac{r_im_i\beta_i^{(I)}-c_i\alpha_i\gamma_i(1-\beta_i^{(II)})}{\tilde r_i}
\le\beta_i^{(I)},
\]
the planning program in the variables \(x_i,v_i\) is \eqref{eq:mc_reduced}
with objective \(\sum_i\tilde r_i\mu_{g,i}(x_i-\tilde\beta_iv_i)\), the
same capacity constraints, and an arrival cap that remains linear.  Hence
Lemma~\ref{prop:approved_flow_reduction} and Parts~1 and~2 of
Proposition~\ref{thm:multiclass_path} hold verbatim.  The single-class problem
is \eqref{eq:sc_lp} with \(\beta^{(I)}\) replaced by \(\tilde\beta\) in the
objective only; since the constraints, and hence the thresholds
\eqref{eq:sc_thresholds_123}, are unchanged, Proposition~\ref{prop:four_phase} holds verbatim
when \(\tilde\beta>0\), whereas when \(\tilde\beta\le0\) the objective is
nondecreasing in \(v\) once generation binds, so \(v^*=J_{\mathrm{eff}}\)
for all \(n_h\ge n_h^{(1)}\) and the throttling and bypass phases are
replaced by permanent reviewer saturation.  Likewise the direct-only
endpoint of Part~3 of Proposition~\ref{thm:multiclass_path} holds if and only
if \(\tilde\beta_i\ge0\) for every class.  Per maintainer hour, screened
work earns \(r_iq_{\mathrm{acc},i}-c_i(1-q_{\mathrm{acc},i})\gamma_i\)
against \(r_im_i-c_i\alpha_i\gamma_i\) for direct work, more whenever the
reviewer is informative, so the comparison behind Phases~1 and~2 keeps its
sign.
\fi
\ifdefined\MainAppendix\else
\section{Fluid Limit}
\label{ec:proof_fluid_limit}

\begin{proof}{Proof of Theorem~\ref{thm:fluid-limit}.}
Fix \(T>0\) and write \(\bar W^n=W^n/n\).  Use the node notation of
Appendix~\ref{app:state_equations}.  Define total service completions
pathwise by \(D_{g,i}^n=S_{g,i}^n\),
\(D_{r,i}^n=S_{r\to g,i}^n+S_{r\to h_s,i}^n\), and
\(D_{h,i}^n=S_{h_d\to g,i}^n+S_{h_d\to c,i}^n+S_{h_s\to g,i}^n+S_{h_s\to c,i}^n\).
These are sums of outcome splits, not independent primitives.  Their
compensators are the integrals of
\(\mu_{g,i}X_i^n\), \(\mu_{r,i}Y_i^n\), and
\(\mu_{h,i}(Z_{i,d}^n+Z_{i,s}^n)\), respectively; arrival and obsolescence
compensators are \(n\lambda_i t\) and
\(\int_0^t\theta_{\ell,i}Q_{\ell,i}^n(s)\,ds\).

\paragraph{Tightness.}
Capacity bounds make every service clock \(O(nT)\).  Queue balances and
nonnegative admissions give
\begin{align*}
Q_{g,i}^n&\le Q_{g,i}^n(0)+A_i^n+D_{r,i}^n+D_{h,i}^n,\quad
Q_{r,i}^n\le Q_{r,i}^n(0)+D_{g,i}^n,\\
Q_{h_d,i}^n&\le Q_{h_d,i}^n(0)+D_{g,i}^n,\quad
Q_{h_s,i}^n\le Q_{h_s,i}^n(0)+D_{r,i}^n.
\end{align*}
The FSLLN and Assumption~\ref{ass:fluid-init} therefore imply that all
scaled queues are \(O_P(1)\) uniformly on \([0,T]\), so every obsolescence
clock is \(O_P(nT)\).  The same argument bounds all outcome counts and the
controlled routing splits; each admission count is bounded through its node
balance by initial mass plus cumulative input.  Hence every coordinate of
\(\bar{\mathcal X}^n\) is stochastically bounded.

\paragraph{Continuity of limits.}
Primitive events move each physical state coordinate by at most one job,
and the bounded-decisions-per-event condition on admissible policies
bounds control-induced moves per event.
Thus every fluid-scaled jump is \(O(1/n)\).

To rule out an accumulation of such jumps, fix $M<\infty$ and restrict to
the event on which the supremum of the total scaled mass in all four waiting
queues is at most \(M\); the probability of its complement can be made
arbitrarily small uniformly in $n$ by the tightness estimates above. On this
event, the aggregate intensity of arrivals, service completions, and
obsolescence is bounded by $nK_M$ for a constant $K_M$. Hence, uniformly over
stopping times $\tau\le T$ and $0\le h\le\delta$, the number of primitive
events on $(\tau,\tau+h]$ is stochastically dominated by a Poisson random
variable with mean $nK_M\delta$. Since each event causes at most a fixed
number of state and control moves, the scaled oscillation over such an
interval is $O_P(\delta)+o_P(1)$. Letting first $n\to\infty$, then
$\delta\downarrow0$, and finally $M\uparrow\infty$ verifies the localized
modulus condition \citep[Theorem~13.2]{billingsley2013convergence}.  Together
with the endpoint bounds above, this proves \(C\)-tightness.  Prohorov's
theorem now gives a weakly convergent subsequence
\(\bar{\mathcal X}^{n_k}\Rightarrow\bar{\mathcal X}\), whose limit is
continuous; on a Skorokhod representation space the convergence is uniform
on compact intervals.

\paragraph{Identification of the limit.}
The FSLLN and random-time-change theorem identify every uncontrolled limit;
for example,
\[
\bar A_i^n\to\lambda_i(\cdot),\quad
\bar D_{g,i}^n\to\int_0^\cdot\mu_{g,i}x_i(s)\,ds,\quad
\bar B_{\ell,i}^n\to\int_0^\cdot\theta_{\ell,i}q_{\ell,i}(s)\,ds
\]
uniformly on compact intervals, with the analogous service-outcome limits.
For the controlled split, its limit \(s_{g\to r,i}\) is dominated, as a
measure, by the total generation-completion limit \(d_{g,i}\).  The
Radon--Nikodym theorem gives a measurable \(\phi_i(t)\in[0,1]\) such that
\[
s_{g\to r,i}(t)=\int_0^t \phi_i(s)\,d d_{g,i}(s)
=\int_0^t \mu_{g,i}\phi_i(s)x_i(s)\,ds.
\]
Set \(v_i=\phi_ix_i\).  Passing to the limit in the balances and capacity
inequalities gives \eqref{eq:fluid-balance}--\eqref{eq:fluid-rates} and
\eqref{eq:fluid-cap}.
\end{proof}
\fi
\ifdefined\MainAppendix\else
\section{Proof of Theorem~\ref{thm:asymptotic_optimality}}
\label{ec:proof_asymptotic_opt}

We first prove a universal LP upper bound.  We then show that deterministic
route quotas and the stabilization regulator converge globally to the
selected target; pooled largest-deficit service preserves that target after
the switch.  These properties give the matching lower bound.

\begin{lemma}[Universal time-average upper bound]\label{lem:baseline_time_average_ub}
Under Assumptions~\ref{ass:regular_params} and~\ref{ass:fluid-init}, every admissible policy sequence satisfies
\[
\limsup_{T\to\infty}\limsup_{n\to\infty}
\frac{1}{nT}\mathbb{E}\!\left[\sum_i r_i C_i^n(T)\right]\le R^*.
\]
\end{lemma}

\begin{proof}{Proof.}
The argument constructs a feasible reduced-LP point from work that is actually processed by maintainers; this avoids imposing rate stability on arbitrary policies. For fixed $(n,T)$, define the deterministic expected average maintainer-service levels
\[
\hat z_{i,p}^n(T):=\frac{1}{nT}\,
\mathbb E\!\left[\int_0^T Z_{i,p}^n(s)\,ds\right],
\qquad p\in\{d,s\}.
\]
The random time-change representations give the exact expected completion-rate identity
\begin{equation}
\frac{\mathbb E[C_i^n(T)]}{nT}
=\mu_{h,i}\Big[(1-\alpha_i)\hat z_{i,d}^n(T)
+q_{\mathrm{acc},i}\hat z_{i,s}^n(T)\Big].
\label{eq:ec_actual_expert_reward}
\end{equation}

Take sequences realizing the iterated limsup in the lemma and then a subsequence along which all bounded expected average service levels converge. Write the limiting maintainer-service levels as $(z_{i,d},z_{i,s})$ and define
\begin{equation}
v_i:=\frac{\mu_{h,i}}{\mu_{g,i}p_{\mathrm{pass},i}}\,z_{i,s},
\qquad
x_i:=v_i+\frac{\mu_{h,i}}{\mu_{g,i}}\,z_{i,d}.
\label{eq:ec_useful_xv}
\end{equation}
These quantities retain only generation and review effort that feeds maintainer service.

We verify feasibility of $(x_i,v_i)$ for the reduced LP. Summing the queue and in-service balances at the AI-review node and at the two human-approval nodes gives
\begin{align*}
S_{r,i}^n(T)
&\le S_{g\to r,i}^n(T)+Q_{r,i}^n(0)+Y_i^n(0),\\
S_{h_d,i}^n(T)
&\le S_{g\to h_d,i}^n(T)+Q_{h_d,i}^n(0)+Z_{i,d}^n(0),\\
S_{h_s,i}^n(T)
&\le S_{r\to h_s,i}^n(T)+Q_{h_s,i}^n(0)+Z_{i,s}^n(0).
\end{align*}
where each $S$ without an outcome subscript denotes total completions at that node. Divide by $nT$, take expectations, and use the exact compensator means. Assumption~\ref{ass:fluid-init} makes all initial boundary terms vanish when the outer limit $T\to\infty$ is taken. The first inequality and the AI-review capacity imply $\sum_i(\mu_{g,i}/\mu_{r,i})v_i\le n_r$; combining the three inequalities with
$S_{g\to r,i}^n+S_{g\to h_d,i}^n=S_{g,i}^n$ and the generation capacity gives $\sum_i x_i\le n_g$.
Clearly $0\le v_i\le x_i$. Moreover, by \eqref{eq:ec_useful_xv},
\[
\sum_i\frac{\mu_{g,i}}{\mu_{h,i}}
\bigl(x_i-p_{\mathrm{rej},i}v_i\bigr)
=\sum_i(z_{i,d}+z_{i,s})\le n_h,
\]
so the human-approval capacity constraint also holds.

It remains to check arrival feasibility. Let $W_i^n(t)$ be the total class-$i$ population in the network. The system-wide mass balance gives
\[
C_i^n(T)=W_i^n(0)+A_i^n(T)-B_{g,i}^n(T)-B_{r,i}^n(T)
-B_{h_d,i}^n(T)-B_{h_s,i}^n(T)-W_i^n(T)
\le W_i^n(0)+A_i^n(T).
\]
After division by $nT$ and the iterated limit, this yields
\[
\mu_{g,i}(1-\alpha_i)(x_i-\beta_i^{(I)}v_i)\le\lambda_i,
\]
where equality between the left-hand side and the limiting reward rate follows from \eqref{eq:ec_actual_expert_reward}, \eqref{eq:ec_useful_xv}, and
$p_{\mathrm{pass},i}q_{\mathrm{acc},i}
=(1-\alpha_i)(1-\beta_i^{(I)})$.
Thus $(x_i,v_i)$ is feasible for \eqref{eq:mc_reduced}, and its reduced-LP objective equals the limiting expected reward. Maximizing over the reduced LP proves the asserted upper bound by $R^*$.
\end{proof}

\begin{lemma}[Fluid complementarity of virtual reservations]\label{lem:reservation_complementarity}
Consider one virtual reservation in the \(n\)th system, with integer
reservation size \(m_n\), queue \(Q^n\), in-service process \(A^n\),
cumulative input \(G^n\), and cumulative admission \(U^n\).  Suppose
\(m_n/n\to a^*\), reservations are not borrowed, and after each primitive
event each pool makes at most one catch-up admission to a nonempty
coordinate with maximal reservation shortfall.  Along any fluid-limit
subsequence for which
\[
(\bar Q^n,\bar A^n,\bar G^n,\bar U^n):=(Q^n/n,A^n/n,G^n/n,U^n/n)
\Rightarrow (q,a,g,u)
\]
uniformly on compact intervals, the limit satisfies the following
properties.  If \(a^*=0\), then \(a(t)\to0\).  If \(a^*>0\), the
positive-overfill part \((a(t)-a^*)^+\) vanishes after a finite transient.
After a further finite regularization transient, at every regular time,
\[
a(t)<a^* \quad\Longrightarrow\quad q(t)=0\quad\text{and}\quad \dot u(t)=\dot g(t).
\]
\end{lemma}

\begin{proof}{Proof.}
If $a^*=0$, the reservation construction sets $m_n=0$ (or, equivalently, at most an $O(1)$ rounding residue), so no fluid-scale admissions enter that coordinate and its initial in-service mass drains through service completions. Thus $a(t)\to0$. If $a^*>0$ and $a(t)>a^*$ at a regular time, the reservation blocks new admissions to this coordinate above its reserved size. Hence the excess level decreases at service rate until it reaches the reservation level. Since fluid service rates are bounded away from zero on any interval on which $a(t)\ge a^*+\epsilon$, the excess over $a^*$ cannot persist indefinitely; the rounding error in $m_n$ disappears after scaling.

It remains to regularize initial underfill.  If a pool has both positive
waiting mass and positive reservation shortfall, each primitive event gives
one catch-up opportunity to a coordinate with maximal shortfall.
Service-completion events replace the slot they open, while every exogenous
arrival event supplies one extra opportunity in each affected pool.
Because the coordinate set and initial fluid mass are finite, the number of
catch-up admissions needed to remove all initial violations is \(O(n)\),
and the FSLLN for the aggregate arrival process, whose rate
\(\Lambda=\sum_i\lambda_i\) is positive, clears it in finite fluid time.
Thereafter the event rule preserves \(q(t)(a^*-a(t))^+=0\).  This argument
uses only \(O(1)\) admissions per event and therefore does not introduce a
fluid-scale initialization jump.

Finally, fix a regular time \(t\) after regularization with \(a(t)<a^*\).
The displayed complementarity gives \(q(t)=0\).  By continuity the
reservation has positive slack on a neighborhood of \(t\), and every input
there is admitted at its event.  Consequently
\[
\bigl(U^n(t+h)-U^n(t)\bigr)-\bigl(G^n(t+h)-G^n(t)\bigr)=o(n)
\]
uniformly for sufficiently small \(h\ge0\).  Dividing by \(n\), passing to
the limit, and differentiating gives \(\dot u(t)=\dot g(t)\).
\end{proof}

\begin{lemma}[Global fluid convergence during stabilization]\label{lem:baseline_regulator_stability}
Under Assumptions~\ref{ass:regular_params} and~\ref{ass:fluid-init}, every
fluid limit in the stabilization mode of Quota Tracking satisfies,
for each class $i$,
\[
x_i(t)\to x_i^*,\qquad y_i(t)\to y_i^*,\qquad z_{i,d}(t)\to z_{i,d}^*,\qquad z_{i,s}(t)\to z_{i,s}^*,
\]
and these service-coordinate limits and the population bound are uniform over
fluid limits with the deterministic initial state in
Assumption~\ref{ass:fluid-init}.
The corresponding waiting queues converge to zero and
\(q_{g,i}(t)\to q_{g,i}^*\).
\end{lemma}

\begin{proof}{Proof.}
The cumulative quota has discrepancy less than one, so every fluid limit
routes the fraction \(\phi_i^*\).  Fix a class \(i\).
Lemma~\ref{lem:reservation_complementarity} applies to each generation,
AI-review, and human-approval reservation.  Thus zero-target service
coordinates drain to zero, and every positive-target coordinate satisfies
the no-overfill and fluid-complementarity properties stated in that lemma.
We therefore focus below on positive targets.

Write
\[
\phi_i^*:=\begin{cases}v_i^*/x_i^*,&x_i^*>0,\\0,&x_i^*=0,\end{cases}
\qquad
\rho_i:=p_{\mathrm{rej},i}+p_{\mathrm{pass},i}(1-q_{\mathrm{acc},i})
=1-(1-\alpha_i)(1-\beta_i^{(I)}).
\]
The coefficient $\rho_i$ is the probability that a class-$i$ task routed through AI review eventually returns to generation rather than being approved. Consider the return-weighted positive deficit
\[
L_i(t):=(x_i^*-x_i(t))^+
+\rho_i(y_i^*-y_i(t))^+
+\alpha_i(z_{i,d}^*-z_{i,d}(t))^+
+(1-q_{\mathrm{acc},i})(z_{i,s}^*-z_{i,s}(t))^+.
\]
For a positive target, the no-overfill part of
Lemma~\ref{lem:reservation_complementarity} gives a finite time after which
the corresponding service coordinate is at most its target.  The time can
be chosen uniformly over the fluid-limit family because all limits have the
same deterministic initial state.  A zero-target coordinate receives no
fluid-scale admissions and drains exponentially; its absolute error
therefore has a common finite integral.  We may consequently fix a common
time \(\tau_i<\infty\) after which every positive-target coordinate of class
\(i\) is at most its target, and handle zero-target coordinates separately
by this exponential bound.

For \(t\ge\tau_i\), write \(d_x=(x_i^*-x_i)^+\), \(d_y=(y_i^*-y_i)^+\),
\(d_d=(z_{i,d}^*-z_{i,d})^+\), and \(d_j=(z_{i,s}^*-z_{i,s})^+\).
At every regular time, reservation complementarity and the positive-part
chain rule give
\begin{align}
D^+d_y&\le-\mu_{r,i}d_y+\mu_{g,i}\phi_i^*d_x,\label{eq:ec_dy_bound}\\
D^+d_d&\le-\mu_{h,i}d_d+\mu_{g,i}(1-\phi_i^*)d_x,\label{eq:ec_dd_bound}\\
D^+d_j&\le-\mu_{h,i}d_j+\mu_{r,i}p_{\mathrm{pass},i}d_y.\label{eq:ec_dj_bound}
\end{align}
These inequalities include the target boundaries.  For example, when
\(y_i=y_i^*\), the derivative of \(d_y\) is
\(\max\{-\dot y_i,0\}\); if its queue is positive the regulator replaces
departures and this derivative is zero, whereas if the queue is empty all
input is admitted and it is at most
\(\mu_{g,i}\phi_i^*d_x\).  The direct and screened-route approval coordinates are
identical cases.

For the generation coordinate, define its input rate
\[
\gamma_{g,i}(t):=\lambda_i+\mu_{r,i}p_{\mathrm{rej},i}y_i(t)
+\mu_{h,i}\alpha_i z_{i,d}(t)
+\mu_{h,i}(1-q_{\mathrm{acc},i})z_{i,s}(t).
\]
When \(x_i<x_i^*\), its queue is empty and all input is admitted.  At the
boundary \(x_i=x_i^*\), a positive queue replaces departures, while an empty
queue gives
\(D^+d_x=(\mu_{g,i}x_i^*-\gamma_{g,i})^+\); excess input, if any, waits.
Using the positive-part chain rule in both
cases, the target balance
\[
\lambda_i+\mu_{r,i}p_{\mathrm{rej},i}y_i^*
+\mu_{h,i}\alpha_i z_{i,d}^*
+\mu_{h,i}(1-q_{\mathrm{acc},i})z_{i,s}^*
=\mu_{g,i}x_i^*+\theta_iq_{g,i}^*,
\]
with
\(q_{g,i}^*:=[\lambda_i-\mu_{g,i}(1-\alpha_i)(x_i^*-\beta_i^{(I)}v_i^*)]/\theta_i\ge0\)
the target generation queue, and dropping the favorable
term \(-\theta_iq_{g,i}^*\), which is zero for a class served up to its
arrival cap and negative otherwise, yield
\begin{equation}
D^+d_x\le-\mu_{g,i}d_x
+\mu_{r,i}p_{\mathrm{rej},i}d_y
+\mu_{h,i}\alpha_i d_d
+\mu_{h,i}(1-q_{\mathrm{acc},i})d_j.
\label{eq:ec_dx_bound}
\end{equation}
Summing \eqref{eq:ec_dy_bound}--\eqref{eq:ec_dx_bound} with the weights of
\(L_i\), the coefficients on \(d_y,d_d,d_j\) cancel because
\(\rho_i=p_{\mathrm{rej},i}+p_{\mathrm{pass},i}(1-q_{\mathrm{acc},i})\),
and the coefficient on \(d_x\) is
\[
-\mu_{g,i}\bigl[1-\alpha_i(1-\phi_i^*)-\rho_i\phi_i^*\bigr]
=-\mu_{g,i}(1-\alpha_i)(1-\beta_i^{(I)}\phi_i^*)=:-c_i<0,
\]
since $\beta_i^{(I)}<1$ by Assumption~\ref{ass:regular_params} and $\phi_i^*\le1$.  Hence
\[
D^+L_i(t)\le -c_i d_x(t)\ \ (t\ge\tau_i),\qquad
\int_{\tau_i}^{\infty}d_x(t)\,dt\le \frac{L_i(\tau_i)}{c_i}.
\]
Integrating \eqref{eq:ec_dy_bound}--\eqref{eq:ec_dj_bound} gives the explicit
cascade bounds
\begin{align*}
\int_{\tau_i}^{\infty}d_y(t)\,dt
&\le \frac{d_y(\tau_i)}{\mu_{r,i}}
+\frac{\mu_{g,i}\phi_i^*}{\mu_{r,i}}
\int_{\tau_i}^{\infty}d_x(t)\,dt,\\
\int_{\tau_i}^{\infty}d_d(t)\,dt
&\le \frac{d_d(\tau_i)}{\mu_{h,i}}
+\frac{\mu_{g,i}(1-\phi_i^*)}{\mu_{h,i}}
\int_{\tau_i}^{\infty}d_x(t)\,dt,\\
\int_{\tau_i}^{\infty}d_j(t)\,dt
&\le \frac{d_j(\tau_i)}{\mu_{h,i}}
+\frac{\mu_{r,i}p_{\mathrm{pass},i}}{\mu_{h,i}}
\int_{\tau_i}^{\infty}d_y(t)\,dt.
\end{align*}
All fluid state derivatives are bounded by a common constant: service inputs
and outputs are capacity bounded, and the generation queue satisfies a
linear obsolescence bound.  Thus the service coordinates are globally Lipschitz,
so their positive deficits are uniformly continuous.  The finite integral
bounds and the no-overfill property imply, by the standard uniformly
continuous integrability argument, that all four service coordinates
converge to their targets.

Moreover, the same calculations, including the finite overfill transients
and zero-target exponential drains, yield constants \(K_{x,i},K_{y,i},
K_{d,i},K_{j,i}<\infty\), depending only on the deterministic initial state,
the primitive parameters, and the LP target, such that every fluid limit
satisfies
\begin{equation}
\begin{aligned}
\int_0^\infty |x_i(t)-x_i^*|\,dt\le K_{x,i},\quad
\int_0^\infty |y_i(t)-y_i^*|\,dt&\le K_{y,i},\\
\int_0^\infty |z_{i,d}(t)-z_{i,d}^*|\,dt\le K_{d,i},\quad
\int_0^\infty |z_{i,s}(t)-z_{i,s}^*|\,dt&\le K_{j,i}.
\end{aligned}
\label{eq:ec_uniform_integral_bounds}
\end{equation}

The generation queue has bounded input and positive obsolescence, so comparison
with \(\dot q=\Gamma_i-\theta_iq\) gives a common bound.  Downstream queues
are bounded by their balances, reservation complementarity, and the finite
service-coordinate errors in \eqref{eq:ec_uniform_integral_bounds}.
Consequently total population is uniformly bounded, and
\eqref{eq:ec_uniform_integral_bounds} gives a common \(K/T\) bound for each
service-coordinate Cesàro error.

For the queue coordinates, fix a node and write \(a(t)\) for its in-service
level, \(\gamma(t)\) for its total input rate, and
\(w(t)=q(t)+a(t)\).  Adding its queue and service equations gives
\[
\dot w(t)=\gamma(t)-\mu a(t)-\theta q(t).
\]
The target obeys
\(\gamma^*=\mu a^*+\theta q^*\).  Hence
\[
\frac{d}{dt}\{w-(q^*+a^*)\}
=-\theta\{w-(q^*+a^*)\}
+\{\gamma-\gamma^*\}+(\theta-\mu)(a-a^*).
\]
All service coordinates converge by the first part of the proof, so every
upstream-generated input satisfies \(\gamma(t)\to\gamma^*\).
Because \(\theta>0\), variation of constants for this stable scalar equation
gives \(w(t)\to q^*+a^*\).  Together with \(a(t)\to a^*\), this yields
\(q(t)\to q^*\) at every node.
\end{proof}
\begin{lemma}[Quota and effort realization]\label{lem:sparse_quota_realization}
Let \(G_i^n(t)\) count routed class-\(i\) generation completions and
\(H_i^n(t)\) those sent to AI review.  Quota Tracking satisfies
\[
\left|H_i^n(t)-\phi_i^*G_i^n(t)\right|<1
\quad\text{for every }i,n,t.
\]
Select an optimal extreme point by the fixed tie-breaking rule.  If pooled
mode begins with the
scaled service vector at its target and with every cumulative service
deficit equal to zero up to \(O(1)\) integer discrepancies, then every fluid
limit keeps the target service levels:
\begin{align*}
\int_0^t x_i(s)\,ds=x_i^*t,\quad
\int_0^t y_i(s)\,ds=y_i^*t,\\
\int_0^t z_{i,d}(s)\,ds=z_{i,d}^*t,\quad
\int_0^t z_{i,s}(s)\,ds=z_{i,s}^*t.
\end{align*}
\end{lemma}

\begin{proof}{Proof.}
For the \(g\)th routed completion, the quota rule sets
\(H_i^n=\lfloor\phi_i^*g\rfloor\); the first claim follows immediately.
After division by \(n\), the routing discrepancy vanishes uniformly on
compact intervals, so every fluid limit routes exactly the fraction
\(\phi_i^*\).

For a queue coordinate \(c\), define cumulative service effort
\(B_c^n(t)=\int_0^tX_c^n(s)\,ds\) and
\(D_c^n(t)=nb_c^*t-B_c^n(t)\).  Add an idle coordinate with target equal to
unused pool capacity and actual effort equal to cumulative idle-server time.
The augmented targets and actual occupancies both sum to pool capacity, so
the augmented deficits sum to zero.  The lexicographic max-deficit rule
assigns each newly available server to a maximal lag, with the occupancy
gap breaking ties.

Starting from \(o(n)\) deficits, the standard largest-lag interchange
argument gives
\[
\sup_{0\le s\le T}\frac{|D_c^n(s)|}{n}\ \Longrightarrow\ 0
\qquad\text{for every fixed }T.
\]
Indeed, suppose a positive fluid deficit \(d>0\) first appeared at a
coordinate \(c\) with service rate \(\mu_c\) at time \(t\).  Cumulative
input to \(c\) minus cumulative completions equals its queue plus its
obsolescence loss.  Before \(t\) every deficit is \(o(n)\), so every
upstream coordinate has delivered its target input up to \(o(n)\), the
exogenous arrivals obey the functional strong law, and the obsolescence
loss is \(o(n)\) because the queue was \(o(n)\).  Hence the queue of
\(c\) at time \(t\) is at least \(\mu_cd\,n+o(n)>0\): a coordinate with a
positive fluid deficit is nonempty, and the next available capacity in its
pool is assigned to a maximal deficit.  If instead its queue is empty, all
incoming fluid is admitted and the deficit cannot grow.  The argument does
not distinguish a generation coordinate served up to its arrival cap, whose
target queue is zero, from a review or approval coordinate; all are
supplied at exactly their service rate at the target and all hold a
backlog equal to their shortfall.  A generation coordinate of a rationed
class has a positive target queue and is nonempty outright.  Pool
feasibility rules out simultaneous growth of positive deficits across an
entire augmented pool, and a server that went idle is reassigned by the
same rule at the arrival that makes a coordinate nonempty.  The largest
residual service time among \(O(n)\) exponential jobs is \(o(n)\) on each
fixed horizon, so nonpreemption does not affect the fluid conclusion.
Since \(D_c^n/n\Rightarrow0\), every fluid limit satisfies
\(\int_0^tX_c(s)\,ds=b_c^*t\).

At the target service vector, exogenous arrivals and returns supply every
generation coordinate at its target input rate
\(\mu_{g,i}x_i^*+\theta_iq_{g,i}^*\), and the quota supplies every
downstream coordinate at its target input rate.  A positive queue, planned
only for a generation coordinate below its arrival cap, makes that supply
more immediate but is not needed.  Hence no target-positive coordinate
loses supply at the boundary, which gives the displayed target service
levels by continuity.
\end{proof}

\begin{proof}{Proof of Theorem~\ref{thm:asymptotic_optimality}.}
\emph{Upper bound.} This follows from Lemma~\ref{lem:baseline_time_average_ub}.

\emph{Lower bound.}  Select the canonical realization of an optimal reduced
LP solution, so its downstream waiting targets are zero and its maintainer-service
targets satisfy
\[
z_{i,d}^*=\frac{\mu_{g,i}}{\mu_{h,i}}(x_i^*-v_i^*),\qquad
z_{i,s}^*=\frac{\mu_{g,i}}{\mu_{h,i}}p_{\mathrm{pass},i}v_i^*.
\]
Because \(\delta_n\to0\), a fluid trajectory can leave stabilization mode
only when its service vector is at the target.  Before that time,
Lemma~\ref{lem:baseline_regulator_stability} gives global convergence; after
that time, resetting the deficits and
Lemma~\ref{lem:sparse_quota_realization} keep the service vector at target.
Thus every fluid trajectory converges to the selected service target, with
the uniform integral bounds in
\eqref{eq:ec_uniform_integral_bounds}.  The same argument applies after any
hysteresis-triggered return to stabilization.
The final paragraph of
Lemma~\ref{lem:baseline_regulator_stability} gives queue convergence before
a switch.  After a switch, target service levels make each downstream queue
a stable scalar balance around zero and the generation queue a stable balance
around \(q_{g,i}^*\).  This proves the full-state claim.

Fix \(T>0\), put
\[
J_n(T):=\frac1{nT}\mathbb E\!\left[\sum_i r_iC_i^n(T)\right],
\]
and select \(n_k\) so that
\(J_{n_k}(T)\to\liminf_{n\to\infty}J_n(T)\).  Theorem~\ref{thm:fluid-limit} and a
further subsequence yield a continuous fluid limit on \([0,T]\).  The
completion compensators give the exact total-reward identity
\[
J_n(T)=\frac1T\mathbb E\!\left[\int_0^T
\sum_i r_i\mu_{h,i}\big((1-\alpha_i)\bar Z_{i,d}^n(t)
+q_{\mathrm{acc},i}\bar Z_{i,s}^n(t)\big)\,dt\right].
\]
The integrand is uniformly bounded by a constant determined by the finite
class set and maintainer capacity.  Hence convergence of the fluid-scaled service
coordinates implies convergence of expectations along the selected
subsequence:
\[
\liminf_{n\to\infty}J_n(T)
=\frac1T\mathbb E\!\left[\int_0^T
\sum_i r_i\mu_{h,i}\big((1-\alpha_i)z_{i,d}(t)
+q_{\mathrm{acc},i}z_{i,s}(t)\big)\,dt\right].
\]
The uniform integral bounds make the Cesàro average in the last display
converge to the LP reward \(R^*\) as \(T\to\infty\).  Therefore
\begin{align*}
\liminf_{T\to\infty}\liminf_{n\to\infty}
\frac1{nT}\mathbb E\!\left[\sum_i r_iC_i^n(T)\right]
&\ge\sum_i r_i\mu_{h,i}
\big((1-\alpha_i)z_{i,d}^*
+q_{\mathrm{acc},i}z_{i,s}^*\big)\\
&=\sum_i r_i\mu_{g,i}(1-\alpha_i)
(x_i^*-\beta_i^{(I)}v_i^*)=R^*.
\end{align*}
Together with the upper bound, this proves \eqref{eq:asymptotic_opt}.
\end{proof}
\fi

\endgroup
\end{APPENDICES}

\section{Calibration Details}
\label{ec:data}

We use the public SWE-Review-Bench release and Table~8 of
\citet{wang2026swereview} (Claude Opus~4.6).  For the high-, medium-, and
low-quality generator tiers, the
\((N,\mathrm{resolved},\mathrm{unresolved},\mathrm{false\ reject},
\mathrm{false\ approve})\) counts are \((500,361,139,21,101)\),
\((462,235,227,54,43)\), and \((422,116,306,30,23)\), where \(N\) is the
release split size, error counts are those of Table~8, and resolved counts
apply the published rates 72.2\%, 50.9\%, and 27.5\% to \(N\) and round.
We set \(\hat\alpha=\mathrm{unresolved}/N\),
\(\hat\beta^{(I)}=\mathrm{false\ reject}/\mathrm{resolved}\), and
\(\hat\beta^{(II)}=\mathrm{false\ approve}/\mathrm{unresolved}\).
Table~8 uses a 500 base for every tier, but the release lists 462 and 422
medium- and low-tier instances; we use the release counts.  A 500 base
gives \((\hat\alpha,\hat\beta^{(I)},\hat\beta^{(II)})=(0.490,0.212,0.176)\)
and \((0.724,0.217,0.064)\) for the medium and low tiers; single-class
thresholds then move by less than 0.1 maintainer, the reviewed-profile
sequence in Section~\ref{subsec:exp_calibrated_which} is unchanged, and
the maximal regrets of the \(q_{\mathrm{acc}}\)-only and \(\eta\)-only
rules become 3.9\% and 3.9\%.  Reviewer failures and missing patches are
not separated in the published aggregates and are treated as part of the
reported outcome cells.

\end{document}